\documentclass[10pt]{amsart}
\usepackage{amsmath}
\usepackage{paralist}
\usepackage{epsfig} 
\usepackage{epstopdf} 
\usepackage[colorlinks=true,hypertexnames=false]{hyperref}
\hypersetup{urlcolor=blue, citecolor=red}
\allowdisplaybreaks

\usepackage{tikz}
\usepackage{mathrsfs}
\usepackage{array}
\usepackage{amsfonts}
\usepackage{amsthm}
\usepackage{amssymb}
\usepackage{graphicx}
\usepackage{esint}
\usepackage{enumitem}
\usepackage[justification=centering]{caption}
\usepackage{float}
\usepackage{ulem}

\usepackage{mathtools}

 \mathtoolsset{showonlyrefs} %

\newcommand{\RR}{\mathbb{R}}

\newcommand{\tl}{\underline{\Lambda}}
\newcommand{\tu}{\overline{\Lambda}}
\newcommand{\ol}{\underline{\textrm{O}}}
\newcommand{\ou}{\overline{\textrm{O}}}
\newcommand{\ool}{\underline{\mathcal{O}}}
\newcommand{\oou}{\overline{\mathcal{O}}}
\newcommand{\lpp}{\mathcal{L}_{p}}
\newcommand{\lqq}{\mathcal{L}_{q}}
\newcommand{\li}{\mathcal{L}}
\newcommand{\lu}{\mathcal{L}_{1}}
\newcommand{\ld}{\mathcal{L}_{2}}
\newcommand{\iio}{\int_{\Omega}}
\newcommand{\ui}{u_{\infty}}
\newcommand{\vi}{v_{\infty}}
\newcommand{\zi}{z_{\infty}}

\newcommand{\xo}{x_{0}}
\newcommand{\uip}{\ui-\phi}
\newcommand{\unp}{u_{n}-\phi}

\definecolor{darkblue}{rgb}{0.05, .05, .9}
\definecolor{darkgreen}{rgb}{0.1, .65, .1}
\definecolor{darkred}{rgb}{0.8,0,0}

\newcounter{dummy}
\makeatletter
\newcommand\myitem[1][]{\item[(#1)]\refstepcounter{dummy}\def\@currentlabel{#1}} 
\makeatother

\makeatletter
\newcommand{\labeltext}[3][]{ 
	\@bsphack
	\csname phantomsection\endcsname
	\def\tst{#1}
	\def\labelmarkup{\emph}
	\def\refmarkup{\normalfont}
	\ifx\tst\empty\def\@currentlabel{\refmarkup{#2}}{\label{#3}}
	\else\def\@currentlabel{\refmarkup{#1}}{\label{#3}}\fi
	\@esphack
	\labelmarkup{#2}
}
\makeatother

\newtheorem{theorem}{Theorem}[section]

\newtheorem{lemma}[theorem]{Lemma}
\newtheorem{proposition}{Proposition}

\theoremstyle{definition}
\newtheorem{definition}[theorem]{Definition}
\newtheorem{remark}{Remark}

\newcommand{\ep}{\varepsilon}

\begin{document}
	
	\title[Iterations of two obstacle problem]{\bf 
	Monotone iterations of two obstacle problems with different operators
	}

	\author[I. Gonz\' alvez, A. Miranda
		 and J. D. Rossi]{Irene Gonz\' alvez, Alfredo Miranda
		 and Julio D. Rossi}
		 
		 \address{ 
		 I. Gonz\'alvez. Departamento de Matem\'{a}ticas, Universidad Aut\'{o}noma de Madrid,
		Campus de Cantoblanco, 28049 Madrid, Spain.
		\newline
		\texttt{~irene.gonzalvez@uam.es}; 
		\bigskip
		\newline
		 \indent A. Miranda and J. D. Rossi. Departamento de Matem\'{a}ticas, FCEyN, Universidad  de Buenos Aires, Pabell\'{o}n I, Ciudad Universitaria (1428), Buenos Aires, Argentina.  \newline 
		 \texttt{~amiranda@dm.uba.ar,~jrossi@dm.uba.ar }		
}
	
	\date{}

	\maketitle

	\begin{abstract}
		In this paper we analyze iterations of the obstacle problem
		for two different operators. We solve iteratively the obstacle problem
		from above or below for two different differential operators 
		with obstacles given by the previous functions in the iterative process. 
		When we start the iterations with a super or a subsolution 
		of one of the operators
		this procedure generates two monotone 
		sequences of functions that we show 
		that converge to a solution to the two membranes
		problem for the two different operators.		
		We perform our analysis in both the variational
		and the viscosity
		settings. 
	\end{abstract}

%
%


\section{Introduction.}

The main goal in this paper is to find solutions to the two membranes problem
as limits of sequences obtained by iterating the obstacle problem. 
Let us first describe the obstacle problem (from above or below) and the then two membranes
problem (as described in \cite{Caffa2}). 

\subsection{The obstacle problem} 
The obstacle problem is one of the main problems in the mathematical study of variational inequalities and free boundary problems. The problem is to find the equilibrium position of an elastic membrane whose boundary is fixed, and which is constrained to lie above a given obstacle. In mathematical terms,
given an operator $L$ (notice that here we can consider fully nonlinear problems of the form $L u = F(D^2u, Du, u, x)$) that describes the elastic configuration of the membrane, a bounded Lipschitz domain $\Omega$ and a boundary datum $f$, the obstacle problem from below (here solutions are assumed to be above the obstacle) reads as
\begin{equation} \label{obs} 
\begin{cases}
u\geq \phi & \text{in } \Omega,\\
L u\geq 0 & \text{in } \Omega,\\
L u= 0 & \text{in } \{u>\phi\},\\
u=f & \text{on } \partial\Omega,
\end{cases}
\end{equation}
or equivalently
\[
\begin{cases}
\min\{Lu,u-\phi\}=0 & \text{in } \Omega,\\
u=f & \text{on } \partial\Omega.
\end{cases}
\]
We will denote by
$$u=\ol(L,\phi, f)$$
the solution to \eqref{obs}.

When the operator $L$ is in divergence form and is associated to an energy functional $E(u)$  the problem becomes a variational problem
and a solution can be obtained minimizing $E$ in the set of functions in the appropriate
Sobolev space that are above the obstacle and take the boundary datum. 
We refer to Section \ref{sect2} for details. 

The obstacle problem can be also stated as follows: we look for the smallest supersolution of $L$ (with boundary datum $f$) that is above the obstacle. This formulation is quite convenient when dealing with fully nonlinear problems using viscosity solutions. We will assume here that the problem \eqref{obs} has a unique continuous viscosity solution (for general theory of viscosity solutions we refer to \cite{CaCa} and \cite{CIL}). This is guaranteed if $L$ has a comparison principle and one can construct barriers close to the boundary so that the boundary datum $f$ is taken continuously. See Section \ref{sect3}. 

We can also consider the obstacle problem from above
(here solutions are assumed to be below the obstacle)
\begin{equation} \label{obs.2} 
\begin{cases}
v\leq \varphi & \text{in } \Omega,\\
L v\leq 0 & \text{in } \Omega,\\
L v= 0 & \text{in } \{u<\varphi\},\\
v=g & \text{on } \partial\Omega,
\end{cases}
\end{equation}
or equivalently
\[
\begin{cases}
\max\{Lv,v-\varphi\}=0 & \text{in } \Omega,\\
v=g & \text{on } \partial\Omega.
\end{cases}
\]
In this case, the obstacle problem can be viewed as follows: 
we look for the largest subsolution of $L$ (with boundary datum $g$) that is below the obstacle. 
We will denote by
$$v=\ou(L,\varphi, g)$$
the solution to \eqref{obs.2}.

For general references on the obstacle problem (including regularity of solutions, that
in some cases are proved to be optimally $C^{1,1}$) we just mention \cite{cafa}, \cite{cafa2}, \cite{toti}, \cite{petro}, the survey \cite{Mon} and references therein.

\subsection{The two membranes problem}

Closely related to the obstacle problem, one of the systems that attracted the attention 
of the PDE community is the two membranes problem. This problem models
the behaviour of two elastic membranes that are clamped at the boundary of a prescribed domain, they are assumed to be ordered, one membrane is 
above the other,
and they are subject to different external forces
(the membrane that is on top is pushed down and the one that is below is pushed up). 
The main assumption here is that the two membranes do not penetrate each other
(they are assumed to be ordered in the whole domain). This situation can be
modeled by two obstacle problems; 
 the lower membrane acts as an obstacle from below for the free elastic equation
 that describes the location of the upper membrane, while, conversely, the upper 
membrane is an obstacle from above for the equation for the lower membrane. 
When the equations that obey the two membranes have a variational structure
this problem can be tackled using calculus of variations (one aims to minimize
the sum of the two energies subject to the constraint that the functions that describe the position of the membranes are always ordered inside the domain, one is bigger or equal than the other), see \cite{VC}. On the other hand, when the involved equations are not variational the analysis
relies on monotonicity arguments (using the maximum principle). 
Once existence of a solution (in an appropriate sense) is obtained a lot of interesting 
questions arise, like uniqueness, regularity of the involved functions, a description of the contact set, the regularity of the contact set, etc. See \cite{Caffa1,Caffa2,S}, the 
dissertation \cite{Vivas}
and references therein.
More concretely, 
given two differential operators $L_1 (D^2u, D u, u, x)$ and $
L_2 
(D^2v, D v, v,x)$ the mathematical formulation the two membranes problem
(with Dirichlet boundary conditions) is 
the following:

\begin{definition} \label{def1.1}
A pair $(u,v)$ is called a solution to the two membranes problem 
if it solves
\begin{equation}
	\label{ED12FG}
	\left\lbrace 
	\begin{array}{ll}
		\displaystyle \min\Big\{ L_1 (u)(x),(u-v)(x)\Big\}=0, \quad & x\in\Omega, \\[6pt]
	\displaystyle \max\Big\{ L_2 (v)(x),(v-u)(x)\Big\}=0, \quad & 
	x\in\Omega, \\[6pt]
	u(x)=f(x), \qquad 	v(x)=g(x), \quad & x\in\partial\Omega.
	\end{array}
	\right.
	\end{equation}
	With our previous notation, this system can be written as
	\begin{equation*}
		u=\ol(L_1,v,f )\qquad \textrm{and}\qquad
		v=\ou(L_2,u,g).
	\end{equation*}
\end{definition}

Remark that, in general, the two membranes problem as stated in Definition
\ref{def1.1} does not have uniqueness. To see this, just take a solution to $L_1(D^2 u)=0$ 
with $u|_{\partial \Omega} = f$ and $v$ the solution to the obstacle problem for $L_2(D^2 v)$ with $u$ as obstacle from above, that is, $v=\ou(L_2,u, g)$. One can easily check that the pair $(u,v)$ is a solution to the general formulation of the two membranes problem stated Definition \ref{def1.1}. Analogously, one can take a solution to $L_2(D^2 v)=0$
with $v|_{\partial \Omega} = g$ and $u$ the solution to the corresponding obstacle problem,
$u=\ol(L_1,v, f)$, to obtain another solution to the two membranes problem. 
In general, these two pairs do not coincide.
 For example, consider $\Omega$ as the interval $(0,1)$,  $f\equiv 1$, $g\equiv 0$ on $\partial\Omega$ and the operators $L_{1}(u'')=u''-10$ and $L_{2}(v'')=v''+2.$ The solution to $L_{1}(u'')=0$ with $u(0)=u(1)=1$ is $\widehat{u}(x)=- 5x(1-x)+1$ and the solution to $L_{2}(v'')=0$ with $v(0)=v(1)=0$
is $\tilde{v}(x)=x(1-x)$ for $x\in(0,1)$.  Let $\widehat{v}=\ou(L_2,\widehat{u}, g)$ and  $\tilde{u}=\ol(L_1,\tilde{v}, f)$. Both pairs $(\widehat{u},\widehat{v})$ and $(\tilde{u},\tilde{v})$ are solutions to the two membranes problem but they are different.  Since $\widehat{u}$ is a solution and $\tilde{u}$ is a supersolution to $L_{1}$ with the same boundary datum, by the comparison principle, we get $\widehat{u}\leq \tilde{u}$. On the other hand, since $\tilde{u}$ is the solution to the obstacle problem for $L_1(u'')$ with $\tilde{v}$ as obstacle from below,  $\tilde{u}\geq \tilde{v}$. Since $\widehat{u}$ and $\tilde{v}$ are not ordered, we obtain that $\widehat{u}\not\equiv\tilde{u}$.

The two membranes problem for the Laplacian with a right hand side, that is, for $L_1(D^2u)=-\Delta u +h_1$ and $L_2(D^2v)=-\Delta v-h_2$, was considered in \cite{VC} 
using variational arguments. Latter, in \cite{Caffa1} the authors solve the two membranes problem for two different fractional Laplacians of different order (two 
linear non-local operators defined by two different kernels). 
Notice that in this case the problem is still variational. In these cases an extra condition appears, namely, the sum of the two operators vanishes,
$
L_1 (D^2 u) + L_2 (D^2 v) =0$,
inside $\Omega$. Moreover, this extra condition together with the variational structure is used to prove a $C^{1,\gamma}$ regularity result for the solution.

The two membranes problem for a fully nonlinear operator was
studied in \cite{Caffa1, Caffa2,S}.
In particular, in \cite{Caffa2} the authors consider a version of
the two membranes problem for two different fully nonlinear operators, $L_1(D^2 u)$ and
$L_2(D^2 v)$. Assuming that $L_1$ is convex and that
$
L_2(X) = -L_1(-X),$
they prove that solutions are $C^{1,1}$ smooth. 

We also mention that a more general version of the two membranes problem
involving more than two membranes was considered by several authors (see for example  \cite{ARS,CChV,ChV}).

\subsection{Description of the main results} As we mentioned at the beginning of the introduction, our main goal is
to obtain solutions to the two membranes problems as limits of
iterations of the obstacle problem. For iterations of the obstacle problem
in a different context we refer to \cite{BlancR}.

Let us consider two different operators $L_1$ and $L_2$ 
(with boundary data $f$ and $g$ respectively, we assume that $f>g$) and generate
two sequences iterating the obstacle problems from above and below.
Given an initial function 
$v_{0}$, take $u_0$ as the solution to the obstacle problem
from below for $L_1$ with boundary datum $f$ and 
obstacle $v_0$, that is, 
$$
u_{0}=\ol(L_1,v_{0},f ).
$$
Now, we use this $u_0$ as the obstacle from above 
for $L_2$ with datum $g$ and obtain
$$
v_{1}=\ou( L_2,u_{0},g).
$$
We can iterate this procedure (solving the obstacle problem 
for $L_1$ or $L_2$ with boundary data $f$ or $g$ and 
obstacle the previous $u$ or $v$) to obtain two sequences
$\{u_n\}_n$, $\{v_n\}_n$, given by
	\begin{equation*}
		u_{n}=\ol(L_1,v_{n},f ),\qquad
		v_{n}=\ou( L_2,u_{n-1},g).
	\end{equation*}
	
	Our main goal here is to show that, when the initial
	function $v_0$ is a 
	subsolution for $L_2$ with boundary datum $g$, then
	both sequences of functions $\{u_{n}\}_{n}$, $\{v_{n}\}_{n}$
	are nondecreasing sequences that converge
	to a limit pair that gives a solution of the two membranes problem, i.e.,  there exists a pair of functions  $(u_{\infty},v_{\infty})$ such that  
	\begin{equation*}
		u_{n}\to u_{\infty}\qquad \textrm{and}\qquad
		v_{n}\to v_{\infty}
	\end{equation*}
	and the limit functions satisfy
	\begin{equation*}
		u_{\infty}=\ol(L_1,\vi,f )\qquad \textrm{and}\qquad
		v_{\infty}=\ou(L_2,\ui,g).
	\end{equation*}
	
	An analogous result can be obtained when we start the iteration with
	$u_0$ a supersolution to $L_1$ with boundary datum $f$ and consider
	$\{u_n\}_n$, $\{v_n\}_n$, given by
	\begin{equation*}
		u_{n}=\ol(L_1,v_{n-1},f ),\qquad
		v_{n}=\ou( L_2,u_{n},g).
	\end{equation*}
	In this case the sequences are non-increasing sequences and 
	also converge to a solution to the two membranes problem.

We will study this iterative scheme using two very different frameworks
for the solutions.
First, we deal with variational methods and understand solutions in the weak sense
(this imposes that $L_1$ and $L_2$ must be in divergence form)
and next we deal with viscosity solutions (allowing $L_1$ and $L_2$
to be general elliptic operators). 

{\bf Variational operators.} To simplify the notation and the arguments,
when we deal with the problem using variational methods, we will
concentrate in the particular choice of $L_1$ and $L_2$ as two
different $p-$Laplacians; that is, we let  
$$	\lpp(w)=  -\Delta_{p}w+h_{p} \qquad
	\mbox{and} \qquad
	\lqq(w)=-\Delta_{q}w+h_{q},
$$
where $\Delta_{p}w = \mbox{div} (|\nabla w|^{p-2} \nabla w)$,
$\Delta_{q}w = \mbox{div} (|\nabla w|^{q-2} \nabla w)$ are two different 
$p-$Laplacians 
and $h_{p}$, $h_{q}$ are two given functions.
These operators are associated to the energies
$$
E_{p}(w)=\frac{1}{p}\int_{\Omega}|\nabla w|^{p}+\int_{\Omega}h_{p}w \qquad
\mbox{and} \qquad
 E_{q}(w)=\frac{1}{q}\int_{\Omega}|\nabla w|^{q}+\int_{\Omega}h_{q}w.
$$
These energies are naturally well-posed in the Sobolev spaces
$W^{1,p} (\Omega)$ and $W^{1,q} (\Omega)$, respectively.

In this variational framework the obstacle problem can be solved
minimizing the energy in the appropriate set of functions. In fact, the solution to
the obstacle problem from below, $u=\ol(\lpp,\varphi,f )$,
can be obtained as the minimizer of the energy $E_{p}$ among functions
 are constrained to be above the obstacle, that is, 
\begin{equation} \label{min.intro}
E_{p}(u)=\min_{w\in \tl^{\,p}_{f,\varphi}}\Big\{E_{p}(w)\Big\},
\end{equation}
with 	
\begin{equation*}
		\tl^{\,p}_{f,\varphi}=\Big\{ w\in W^{1,p}(\Omega)\,:\;\; w-f\in W^{1,p}_{0}(\Omega),\;\;w\geq \varphi\;\;\textrm{a.e}\;\;\Omega\Big \}.
	\end{equation*}
	The minimizer in \eqref{min.intro} exists and is unique provided the set
	$\tl^{\,p}_{f,\varphi}$ is not empty.
	We refer to Section \ref{sect2} for details and precise definitions. 
	
	Due to the fact that the problem satisfied by the limit of the sequences
	constructed iterating the obstacle problem from above and below does
	not have uniqueness (see the comments after Definition \ref{def1.1}) we need
	to rely on monotonicity of the sequences to obtain convergence (rather than use
	compactness arguments that provide only convergence along subsequences). 
	
	We notice at this point that, in this variational setting, one can solve the two membranes problem just minimizing the total energy 
	$
	E(u,v) = E_p (u) + E_q(v)
	$
	in the subset of $W^{1,p}(\Omega)\times W^{1,q}(\Omega)$ given by $\{(u,v) : u-f\in W^{1,p}_{0}(\Omega),
	v-g\in W^{1,q}_{0}(\Omega), u\geq v \textrm{ a.e }\Omega \}$.
	The fact that there is a unique pair that minimizes $E(u,v)$ follows from the strict 
	convexity of the functional using the direct method of calculus of variations. One can check that the minimizing pair is in fact a solution to the two membranes problem
	given in Definition \ref{def1.1}. However, in general, there are other solutions
	to the two membranes problem in the sense of Definition \ref{def1.1} that are not
	minimizers of $E(u,v)$.
	Since solutions according to
		Definition \ref{def1.1} are, in general, not unique, we observe that 
	the limit that we prove to exist for the sequences that we construct iterating the obstacle problem from above and below does not necessarily converge to the unique minimizer to
	$E(u,v)$.

{\bf Fully nonlinear operators.}
On the other hand, when we deal with viscosity solutions to the obstacle \mbox{problems}
we use the general framework described in \cite{CIL} and we understand sub and supersolutions applying the operator to a smooth test function that touches the graph of the sub/supersolution from above or below at some point in the domain. 
Remark that here we can consider fully nonlinear
elliptic operators that need not be in divergence form. 

When we develop the viscosity theory of the iterations 
in this general viscosity setting it is more difficult to obtain estimates
in order to pass to the limit in the sequences $\{u_n\}_n$,
$\{v_n\}_n$ and hence we have to rely again on monotonicity
(that is obtained using comparison arguments). 

In the viscosity framework there is no gain in 
considering two particular operators. Therefore, we will
state and prove the results for two general second order elliptic operators that
satisfy the comparison principle.
However, if one looks for a model problem one can think on two
normalized $p-$Laplacians, that is, let 
$$L_1(w)=  - \beta_1 \Delta w - \alpha_1 \Delta_\infty w +h_{p} \quad
	\mbox{and} \quad
	L_2(w)= - \beta_2 \Delta w - \alpha_2 \Delta_\infty w +h_{q}.
$$
Here $\Delta w = trace (D^2w)$ (the usual Laplacian),
$\Delta_\infty w = 
\langle D^2 w \frac{\nabla w}{|\nabla w|}, \frac{\nabla w}{|\nabla w|} \rangle $
(the normalized infinity Laplacian), $\alpha_i, \beta_i$ are nonnegative
coefficients 
and $h_p$, $h_q$ are continuous functions.
This operator, the normalized $p-$Laplacian, appears naturally when one considers
game theory to solve nonlinear PDEs, we refer to 
\cite{MPRb,PSSW,PS} and the recent books \cite{BRLibro,Lewicka}. 

Let us finish the introduction with a short paragraph
concerning the equivalence between weak and viscosity solutions for divergence form operators. 
For the Dirichlet problem  $\Delta u =0$ the notions of weak and viscosity solutions coincide (and in fact the Dirichlet problem
has a unique classical solution), see \cite{Hi} and \cite{RWZ}.
Moreover, the equivalence between weak and
viscosity solutions include quasi-linear equations, \cite{JLM,Medina}, and some non-local equations,
\cite{BM,dBdPO}.

\medskip

The paper is organized in two sections; in Section \ref{sect2}
we deal with that variational setting of the problem and in Section
\ref{sect3} we analyze the problem using viscosity theory.


\section{Variational solutions}
\label{sect2}

In this section we will obtain two sequences of solutions to variational obstacle problems whose limits constitute a pair of functions that is a variational solution to the two membranes problem. 

Let us start by introducing some notations and definitions. We follow 
\cite{libroFinlan}. For short, we write   
\begin{align*}
	&\lpp(w)=  -\Delta_{p}w+h_{p}, &E_{p}(w)=\frac{1}{p}\int_{\Omega}|\nabla w|^{p}+\int_{\Omega}h_{p}w,\\&\lqq(w)=-\Delta_{q}w+h_{q},
	&E_{q}(w)=\frac{1}{q}\int_{\Omega}|\nabla w|^{q}+\int_{\Omega}h_{q}w,
\end{align*}
for the operators and their associated energies. In this variational context, we 
now introduce the definition of weak super and subsolutions.

\begin{definition}
		Given $p\in (1,\infty)$ and $h_{p}\in W^{-1,p}(\Omega)$ we say that $u\in W^{1,p}(\Omega)$ is a weak supersolution (resp. subsolution) for $\lpp$ if 
			\begin{equation*}
				\int_{\Omega}|\nabla u|^{p-2}\nabla u\cdot\nabla w+\int_{\Omega}h_{p}w\geq 0 \ \ (\mbox{resp.} \le 0 )
			\end{equation*} for all non-negative $w\in W^{1,p}_0(\Omega)$ and in this case we writte $\mathcal{L}_{p}u\geq 0$ (resp. $\mathcal{L}_{p}u\leq 0$) weakly in $\Omega$. We say that $u\in W^{1,p}(\Omega)$ is a weak solution for $\lpp$ if 
		\begin{equation*}
		\int_{\Omega}|\nabla u|^{p-2}\nabla u\cdot \nabla w+\int_{\Omega}h_{p}w= 0
	\end{equation*} for all $w\in W^{1,p}_0(\Omega)$ and we write $\mathcal{L}_{p}u= 0$ weakly in  $\Omega$.
\end{definition}
\begin{definition}
	Given $p,\,q\in (1,\infty)$, let $f\in W^{1,p}(\Omega)$,  $\varphi\in L^{1}(\Omega)$, $h_{p}\in W^{-1,p}$ and 
	\begin{equation*}
		\tl^{\,p}_{f,\varphi}=\Big\{ w\in W^{1,p}(\Omega)\,:\;\; w-f\in W^{1,p}_{0}(\Omega),\;\;w\geq \varphi\;\;\textrm{a.e}\;\;\Omega\Big \}.
	\end{equation*}
 When $\tl^{\,p}_{f,\varphi} \neq \emptyset$, we say $u\in 	\tl^{\,p}_{f,\varphi}$ is the
  solution of the $p\,$-Laplacian upper obstacle problem in a varational sense with obstacle $\varphi$ and boundary datum $f$ if 
\begin{equation*} 
E_{p}(u)=\min_{w\in \tl^{\,p}_{f,\varphi}}\Big\{E_{p}(w)\Big\},
\end{equation*}
i.e, $u$ minimizes the energy associated with the operator $\lpp:=-\Delta_{p}+h_{p}$
in the set of functions that are above the obstacle and take
the boundary datum, $\tl^{\,p}_{f,\varphi}$. In this case, we denote  $u=\ol(\lpp,\varphi, f)$.
In Proposition \ref{prop0} we show that this solution exists and is unique
provided the set $\tl^{\,p}_{f,\varphi}$ is not empty.

Analogously, for $g\in W^{1,q}(\Omega)$,  $\varphi\in L^{1}(\Omega)$, $h_{q}\in W^{-1,q}$, we define
	\begin{equation*}
	\tu^{\,q}_{g,\varphi}=\Big\{ w\in W^{1,q}(\Omega)\,:\;\; w-g\in W^{1,q}_{0}(\Omega),\;\;w\leq \varphi\;\;\textrm{a.e}\;\;\Omega\Big \}.
\end{equation*}
 When $\tu^{\,q}_{g,\varphi} \neq \emptyset$, we say that  
$v\in\tu^{\,q}_{g,\varphi}$ is the solution of the $q\,$-Laplacian lower obstacle problem in a varational sense with obstacle $\varphi$ and  boundary $g$ data if 
\begin{equation*}
	E_{q}(v)=\min_{w\in \tu^{\,q}_{g,\varphi}}\Big\{E_{q}(w)\Big\}.
\end{equation*} 
 For short, we denote $v$ as $v=\ou(\lqq,\varphi,g)$.
\end{definition}
\begin{proposition} \label{prop0}
	Let be $h_{p}\in W^{-1,p}(\Omega)$ and $f\in W^{1,p}(\Omega)$. If $	\tl^{\,p}_{f,\varphi}\neq\emptyset$ there exists a unique $u\in\tl^{\,p}_{f,\varphi}$ such that $E_{p}(u)\leq E_{p}(w)$ for all $w\in\tl^{\,p}_{f,\varphi}$, i.e, there exists a unique $u\in\tl^{\,p}_{f,\varphi}$ that minimizes the energy $E_p$
	in $\tl^{\,p}_{f,\varphi}$. That is, $u=\ol(\lpp,\varphi, f)$.
 This function satisfies
	\begin{equation*}
		\begin{cases}
			u\geq \varphi&\;\;\textrm{a.e. in}\;\;\Omega,\\
			\li_p u\geq 0&\;\;\textrm{weakly in}\;\;\Omega,\\
			\li_p u=0&\;\;\textrm{weakly in}\;\;\{ u>\varphi\} \mbox{ (for $\varphi$ continuous)},\\
			u=f&\;\;\textrm{a.e. in}\;\;\partial\Omega.
		\end{cases}
	\end{equation*}
		
	Analogously, let $h_{q}\in W^{-1,q}(\Omega)$ and $g\in W^{1,q}(\Omega)$, if $\tu^{\,q}_{g,\varphi}\neq\emptyset$, there exists a unique $v\in\tu^{\,q}_{g,\varphi}$ such that $E_{q}(v)\leq E_{q}(w)$ for all $w\in\tu^{\,q}_{g,\varphi}$. That is, $v=\ou(\lqq,\varphi,g)$.  This function satisfies
	\begin{equation*}
		\begin{cases}
			v\leq \varphi&\;\;\textrm{a.e. in}\;\;\Omega,\\
			\li_q v\leq 0&\;\;\textrm{weakly in}\;\;\Omega,\\
			\li_q v=0&\;\;\textrm{weakly in}\;\;\{ v<\varphi\} \mbox{ (for $\varphi$ continuous)},\\
			v=g&\;\;\textrm{a.e. in}\;\;\partial\Omega.
		\end{cases}
	\end{equation*}
\end{proposition}

\begin{proof}
	Let us prove existence and uniqueness of the solution to the obstacle problem from below. 
	The proof is contained in Theorem 3.21 in \cite{libroFinlan} but we reproduce some details here for completeness. Suppose that $\tl^{\,p}_{f,\varphi}\neq\emptyset$. Since $E_p$ is a coercive and strictly convex functional in $W^{1,p}(\Omega)$, we obtain that all minimizing sequences $u_j\in\tl^{\,p}_{f,\varphi}$ converge weakly to the same limit $u\in W^{1,p}(\Omega)$
	(that we want to show that is the unique minimizer of $E_p$ in $ \tl^{\,p}_{f,\varphi}$). Let us show that $u\in \tl^{\,p}_{f,\varphi}$. Rellich-Kondrachov's Theorem implies that, taking a subsequence, $u_j\rightarrow u$ strongly in $L^p(\Omega)$. Then, taking a subsequence, we have that $u_j\rightarrow u$ a.e. in $\Omega$. Thus, since $u_j\geq \varphi$ a.e. $\Omega$ and $u_j-f\in W^{1,p}_0(\Omega)$ for all $j\geq 1$, we get $u\geq \varphi$ a.e. $\Omega$ and $u-f\in W^{1,p}_0(\Omega)$.

Let us prove that $u$ verifies the properties stated in the proposition. Let us start with the fact that $\li_p u\geq 0$ weakly in $\Omega$. Since $\tl^{\,p}_{f,\varphi}$ is a convex set, $u+t(v-u)\in\tl^{\,p}_{f,\varphi}$ for all $ v\in\tl^{\,p}_{f,\varphi}$ and $t\in[0,1]$. Thanks to the fact that $E_{p}$ reaches its minimum in the set $\tl^{\,p}_{f,\varphi}$ at $u$, we have that 
		\begin{equation}
		i(t):=E_p[u+t(v-u)]
	\end{equation}
satisfies $i(t)\geq i(0)$ and therefore $i'(0)\geq 0$. Let us compute
\begin{equation}
	i'(t)=\int_{\Omega}[|\nabla u+t\nabla(v-u)|^{p-2}(\nabla u+t\nabla(v-u))\cdot\nabla(v-u)-h_p(v-u)]dx.
\end{equation}
Taking $t=0$, we get 
\begin{equation}
	\label{iprima0}
	i'(0)=\int_{\Omega}[|\nabla u|^{p-2}\nabla u\cdot\nabla(v-u)-h_p(v-u)]dx\geq 0.
\end{equation}
If we consider $w\in W^{1,p}_0(\Omega)$ $w\geq 0$ a.e. $\Omega$, and $v=u+tw\in\tl^{\,p}_{f,\varphi}$ for $t\geq 0$ small enough, if we come back to \eqref{iprima0}, we obtain
 \begin{equation}
 	t\int_{\Omega}[|\nabla u|^{p-2}\nabla u\cdot\nabla w-h_p w]dx\geq 0
 \end{equation}
with $t\geq 0$. This implies that $$\li_p u\geq 0$$ weakly in $\Omega$. 

Finally, we show that $$\li_p u= 0$$ weakly in $\{u(x) >\varphi(x)\}$ when $\varphi$ is continuous. From Section 3.26 in \cite{libroFinlan} we have that the solution to the obstacle problem with continuous obstacle becomes also continuous (after a redefinition in a set of measure zero). Therefore, for a continuous obstacle the set $\{x:u (x) > \varphi (x) \}$ is an open set.
Take a ball 
$$
\overline{B}_r  \subset \{x:u (x) > \varphi (x) \},
$$
and consider a nonnegative $w\in W^{1,p}_0(B_r) \cap C_0(\overline{B_r})$. Using that $\varphi< u$ in $\partial B_r$, we get $v=u+tw\in\tl^{\,p}_{f,\varphi}$ for $|t|$ small enough. Then, let us come back to \eqref{iprima0} and obtain
\begin{equation}
	t\int_{B_r}[|\nabla u|^{p-2}\nabla u\cdot\nabla w-h_p w]dx\geq 0.
\end{equation}
In this case $t$ can be positive or negative. Thus, we get
\begin{equation}
	\int_{B_r}[|\nabla u|^{p-2}\nabla u\nabla w-h_p w]dx= 0,
\end{equation}
wich implies $$\li_p u= 0$$ weakly in $B_r$ and therefore in $\{u>\varphi\}$.

The other case (an obstacle from above) is analogous.
\end{proof}

This minimizer of the energy is in fact the infimum of weak supersolutions
are above/below the obstacle.

\begin{proposition} \label{prop1}
	Let $u=\ol(\lpp,\varphi, f)$. This function satisfies
	\begin{equation}
	\label{Remark1.1}	u=\inf\Big\{w\in \tl^{\,p}_{f,\varphi}\,:\;\;\lpp w\geq 0\;\;\textrm{weakly in}\;\;\Omega\Big\},
	\end{equation}
 
Analogously, let $v=\ou(\lqq,\varphi,g)$, then
	\begin{equation}
	\label{Remark1.2}	v=\sup\Big\{w\in \tu^{\,q}_{g,\varphi}\,:\;\;\lqq w\leq 0\;\;\textrm{weakly in}\;\;\Omega\Big\}.
\end{equation}
\end{proposition}
\begin{proof}
	Let us consider $$\overline{u}=\inf\Big\{w\in \tl^{\,p}_{f,\varphi}\,:\;\;\lpp w\geq 0\;\;\textrm{weakly in}\;\;\Omega\Big\}.$$
	We will prove that $\overline{u}=u$ in $\Omega$.
	
	Let us start proving that $u\geq \overline{u}$. By Proposition \ref{prop0}, we have
	$$
	u\in\Big\{w\in \tl^{\,p}_{f,\varphi}\,:\;\;\lpp w\geq 0\;\;\textrm{weakly in}\;\;\Omega\Big\},
	$$
	and then we get $$u\geq \overline{u}=\inf\Big\{w\in \tl^{\,p}_{f,\varphi}\,:\;\;\lpp w\geq 0\;\;\textrm{weakly in}\;\;\Omega\Big\}.$$
	
	Now, let us prove that $u\leq \overline{u}$. Given $w\in\tl^{\,p}_{f,\varphi}$ such that $\lpp w\geq 0$ weakly in $\Omega$, we have
	\begin{equation}
			\int_{\Omega}[|\nabla w|^{p-2}\nabla w\cdot\nabla v-h_p v]dx\geq 0 \quad \mbox{for all} \  v\in W^{1,p}_0(\Omega)
	\end{equation}
and then, using \eqref{iprima0}, we get
\begin{equation}
	\label{des2}
	\int_{\Omega}[|\nabla u|^{p-2}\nabla u\cdot\nabla(v-u)-h_p(v-u)]dx\geq 0 \quad \mbox{for all} \ v\in\tl^{\,p}_{f,\varphi} (\Omega).
\end{equation}
Let us consider $z=\min\{w,u\}\in\tl^{\,p}_{f,\varphi}$, then $u-z\in W^{1,p}_0(\Omega)$, $u-z\geq 0$. Then
\begin{equation}
0\leq\int_{\Omega}[|\nabla w|^{p-2}\nabla w\cdot\nabla (u-z)-h_p (u-z)]dx 
\end{equation}
and, using  \eqref{des2}, we obtain
\begin{equation}
0\geq\int_{\Omega}[|\nabla u|^{p-2}\nabla u\cdot\nabla(u-z)-h_p(u-z)]dx.
\end{equation}
Substracting these inequalities we conclude that
\begin{equation}
	0\leq\int_{\Omega}[(|\nabla w|^{p-2}\nabla w-|\nabla u|^{p-2}\nabla u)\cdot\nabla (u-z)]dx, 
\end{equation}
which implies 
\begin{equation}
	0\leq\int_{\{u>z\}}[(|\nabla w|^{p-2}\nabla w-|\nabla u|^{p-2}\nabla u)\cdot\nabla (u-w)]dx 
\end{equation}
using that $(|b|^{p-2}b-|a|^{p-2}a)(a-b)\leq 0$ (with a strict inequality for $a\neq b$) we obtain $|\{u>z\}|=0$. Then, $u=z\leq w$. Thus, we have obtained that $u\leq \overline{u}$.

The other case is analogous.
\end{proof}

\begin{remark}
{\rm Notice that Proposition \ref{prop1} implies that the solution to the obstacle problem
is monotone with respect to the obstacle in the sense that, for $\varphi_1\geq \varphi_2$, we have $u_1=\ol(\lpp,\varphi_1, f) \geq u_2=\ol(\lpp,\varphi_2, f)$.
}
\end{remark}

Now, we are ready to introduce the definition of a weak solution to the 
two membranes problem.

\begin{definition}
	\label{df1}
	Let $f\in W^{1,p}(\Omega)$ and $g\in W^{1,q}(\Omega)$ be two functions such that $f\geq g$ in $\partial\Omega$ in the sense of traces. Take $h_p\in W^{-1,p}(\Omega)$ and $h_{q}\in W^{-1,q}(\Omega)$. We say that the pair $(u,v)$ with $u\in \tl^{\,p}_{f,v} $, $v\in \tu^{\,q}_{g,u}$ is a solution of the two membranes problem if 
    \begin{equation*}
    	(u,v) \;\;\textrm{satisfies} \;\;
    \begin{cases}
    	\displaystyle E_{p}(u)=\min_{w\in \tl^{\,p}_{f,v}}\Big\{E_{p}(w)\Big\},\\
    	\displaystyle E_{q}(v)=\min_{w\in \tu^{\,q}_{g,u}}\Big\{E_{q}(w)\Big\},
    \end{cases}
    \end{equation*}
i.e.,
\begin{equation*}
	u=\ol( \lpp,v,f) \;\;\textrm{and}\;\; v=\ou( \lqq,u,g),
	\end{equation*}
\end{definition}

 \begin{remark}
 	In general, given a pair $(f,g)$, the solution to the two membranes problem
	is not unique. 
 \end{remark}

\subsection{Iterative method. }

The main result of this section reads as follows.

\begin{theorem}\label{prop2}
	For $p\geq q $ let 
	$f\in W^{1,p}(\Omega)$ and $g\in W^{1,q}(\Omega)$ be two functions such that $f\geq g$ in $\partial\Omega$ in the sense of traces and take $h_p\in W^{-1,p}(\Omega)$ and $h_{q}\in W^{-1,q}(\Omega)$. Let us
	consider $v_{0}$ a weak subsolution of $ \lqq v=0$ in $\Omega$  such that $v_{0}-g\in W^{1,q}_{0}(\Omega)$ and then define inductively the sequences
	\begin{equation*}
		u_{n}=\ol(\lpp,v_{n},f ),\qquad
		v_{n}=\ou( \lqq,u_{n-1},g).
	\end{equation*}
	
	Both sequences of functions $\{u_{n}\}_{n=0}^{\infty}\subset W^{1,p}(\Omega)$, $\{v_{n}\}_{n=0}^{\infty}\subset W^{1,q}(\Omega)$ converge strongly in $W^{1,p}(\Omega)$ and $W^{1,q}(\Omega)$, respectively. Moreover, the limits of the sequences are a solution of the two membranes problem. That is, there exists a pair of functions  $u_{\infty}\in\tl^{\,p}_{f,\vi}$ and $v_{\infty}\in \tu^{\,q}_{g,\ui}$ such that  
	\begin{equation*}
		u_{n}\to u_{\infty}\;\;\textrm{strongly in} \;\;W^{1,p}(\Omega)\;\;\textrm{and}\;\;v_{n}\to v_{\infty}\;\;\textrm{strongly in}\;\; W^{1,q}(\Omega),
	\end{equation*}
	and, in addition, the limit pair $(u_{\infty},v_{\infty})$ satisfies
	\begin{equation*}
		u_{\infty}=\ol(\lpp,\vi,f )\;\;\textrm{and}\;\;
		v_{\infty}=\ou(\lqq,\ui,g).
	\end{equation*}
\end{theorem}
%
%

\begin{proof}	
	Notice that from the definition of the sequences as solutions to
	the obstacle problem we have $$v_{n}\leq u_{n}$$ a.e. $\Omega$ for each $n\in\mathbb{N}$. Since we assumed that the boundary data are ordered, $f\geq g$, and functions in $W^{1,p}(\Omega)$ and in $W^{1,q}(\Omega)$ have a trace on the boundary, we can also say that $u_n\geq v_n$ a.e. on the boundary (with respect to the $(N-1)$-dimensional measure). 
	From the construction of the sequences we obtain that the sets $\tl^{\,p}_{f,v_n}$ and $ \tu^{\,q}_{g,u_n}$ are not empty, 
	in fact we have that $u_{n-1} \in \tl^{\,p}_{f,v_n}$ (since $v_n$ solves an obstacle problem from above with obstacle $u_{n-1}$, we have $v_n \leq u_{n-1}$)
	and $v_n \in \tu^{\,q}_{g,u_n}$ (since $u_n$ solves an obstacle problem from below with obstacle $v_{n}$, we have $v_n\leq u_n$).
Hence the sequences $\{u_{n}\}$ and $\{v_{n}\}$ are well defined.
	
	We will divide our arguments in several steps. 

{\bf First step}: First, we prove a monotonicity result.
	The two sequences are non-decreasing.

Let us see that $v_{0}\leq v_{1}$ a.e. $\Omega$. This is due to the fact that $v_{0}$ is a weak subsolution of $\lqq v=0$ and $v_{0}\in\tu^{q}_{g,u_{0}}$ and by \eqref{Remark1.2} we have
\begin{equation*}
v_{1}=\sup\Big\{w\in\tu^{q}_{g,u_{0}}\,:\;\;  \lqq w\leq 0\;\; \textrm{weakly in}\;\;\Omega\Big\}.
\end{equation*}
Then, we conclude that $v_0\leq v_1$. 

Also we can see that $u_{0}\leq u_{1}$ a.e. $\Omega$. In fact, from \eqref{Remark1.1},
we have
\begin{equation*}
	u_{0}=\inf\Big\{w\in\tl^{p}_{f,v_{0}}\,:\;\; \lpp w\geq 0 \;\;\textrm{weakly in}\;\;\Omega\Big\}
\end{equation*}
and then, using that $v_{1}\geq v_{0}$ a.e. $\Omega$, we obtain that $u_{1}\in\tl^{q}_{f,v_{0}}$ and $u_{1}$ is a weak supersolution for $\lpp u=0$. Then, $u_0\le u_1$ a.e. in $\Omega$.

Now, to obtain the general case, we just use an inductive argument. 
Suppose $v_{n-1}\leq v_{n}$ and $u_{n-1}\leq u_{n}$  a.e. $\Omega$.  Since $v_{n}\leq u_{n}$, we have $v_{n}\in \tu^{q}_{g,u_{n}}$. From \eqref{Remark1.2} we have
\begin{equation*}
	v_{n+1}=\sup \Big\{w\in \tu^{q}_{g,u_{n}}\,:\;\;\lqq w\leq 0\;\;\textrm{weakly in}\;\;\Omega\Big\}\;\;\textrm{and}\;\;\lqq v_{n}\leq 0\;\;\textrm{weakly in}\;\;\Omega
\end{equation*}
which implies $v_{n}\leq v_{n+1}$ a.e. $\Omega$. On the other hand, the last inequality and the fact that  $v_{n+1}\leq u_{n+1}$ give us $u_{n+1}\in \tl^{p}_{f,v_{n}}$. Hence, \eqref{Remark1.1} implies that
\begin{equation*}
	u_{n}=\inf \Big\{w\in \tl^{p}_{f,v_{n}}\,:\;\;\lpp w\geq 0\;\;\textrm{weakly in}\;\;\Omega\Big\}\;\;\textrm{and}\;\;\lpp u_{n+1}\geq 0\;\;\textrm{weakly in}\;\;\Omega,
\end{equation*}
and we conclude that $u_{n}\leq u_{n+1}$ a.e. $\Omega$.

{\bf Second step}: Our next step is to show that the sequences $\{u_{n}\}_{n=1}^{\infty}$ and $\{v_{n}\}_{n=1}^{\infty}$ are bounded in $W^{1,p}(\Omega)$ and $W^{1,q}(\Omega)$, respectively. To prove this fact let us take $w\in W^{1,q}(\Omega)$ the solution to $\lqq w=0$ with boundary datum $g$. That is, $w$ satisfies
\begin{equation*}
	\lqq w=0\;\;\textrm{ weakly in}\;\;\Omega \;\;\textrm{and}\;\;w=g\;\;\textrm{in}\;\;\partial\Omega\;\;\textrm{in the sense of traces.}
\end{equation*}
 By \eqref{Remark1.2}, we have that 
 $\lqq v_{n}\leq 0$ weakly in $\Omega$, and thanks to $v_{n}-w\in W^{1,q}_{0}(\Omega)$ and the comparison principle for $\lqq$, we obtain that $$v_{n}\leq w$$ a.e. $\Omega$. 
 Let us consider $z=\ol(\lpp, w,f)$. Since $v_{n}\leq w$ we have  $z\in\tl^{p}_{f,v_{n}}$ for all $n\in\mathbb{N}$. Now, since $u_n$ is the solution to the obstacle
 problem, $u_{n}$ minimizes the energy among functions
 in the set $\tl^{p}_{f,v_{n}}$. Then, as $z$ is in the former set, we get $$ E_{p}(u_{n})\leq E_{p}(z)$$ for all $n\in\mathbb{N}$, that is, we have
\begin{equation*}
	\iio\frac{|\nabla u_{n}|^{p}}{p}+\iio h_{p}u_{n}\leq \iio\frac{|\nabla z|^{p}}{p}+\iio h_{p}z.
\end{equation*}
We can rewrite this inequality to obtain
\begin{equation}
	\label{ret1}
	\iio|\nabla u_{n}|^{p}\leq C(z,p,h_p)+C(p,h_p)\lVert u_n\rVert_{L^p(\Omega)}.
\end{equation}
Using that $u_{n}-f\in W^{1,p}_{0}(\Omega)$, from Poincare's inequality, we obtain 
\begin{equation*}
	\|u_{n}-f\|_{L^{p}(\Omega)}\leq C \|\nabla u_{n}-\nabla f\|_{L^{p}(\Omega)} \;\;\textrm{for all}\;\;n\in\mathbb{N}.
\end{equation*}
Then, by the triangle inequality we get
\begin{equation*}
	\|u_{n}\|_{L^{p}(\Omega)}\leq C(f)+ C \|\nabla u_{n}\|_{L^{p}(\Omega)} \;\;\textrm{for all}\;\;n\in\mathbb{N}.
\end{equation*}
If we come back to \eqref{ret1}, we obtain
\begin{equation*}
	\iio|\nabla u_{n}|^{p}\leq C(z,p,h_p,f)+C(p,h_p,f)\lVert \nabla u_n\rVert_{L^p(\Omega)}.
\end{equation*}
Now, if we use Young's inequality $ab\leq \varepsilon a^p+C(\varepsilon)b^{p'}$ with $\varepsilon=\frac{1}{2}$ in the last term, we conclude
\begin{equation*}
	\iio |\nabla u_{n}|^{p}\leq C(z,p,h_p,f)\;\;\textrm{for all}\;\;n\in\mathbb{N}.
\end{equation*}
Thus, we get the desired bound for $u_n$, there exists
a constant $C$ independent of $n$ such that
\begin{equation*}
	\|u_n\|_{W^{1,p}(\Omega)}\leq C \ \ \textrm{for all}\;\;n\in\mathbb{N}.
\end{equation*}

Now, we prove that $\{v_{n}\}$ is bounded in $W^{1,q}(\Omega)$. Since $v_{0}\leq v_{n}\leq u_{n-1}$ we have that $v_{0}\in \tl^{q}_{g,u_{n-1}}$ for all $n\in\mathbb{N}$. This implies
$$E_{q}(v_{n})\leq E_{q}(v_{0}).$$ That is, 
\begin{equation*}
	 \iio\frac{|\nabla v_{n}|^{q}}{q}+\iio h_{q}v_{n}\leq \iio\frac{|\nabla v_{0}|^{q}}{q}+\iio h_{q}v_{0}.
\end{equation*}
Now, we just proceed as before to obtain
\begin{equation*}
	\|v_n\|_{W^{1,q}(\Omega)}\leq C \ \ \textrm{for all}\;\;n\in\mathbb{N}.
\end{equation*}

{\bf Third step}: We show strong convergence to a solution of the two membranes problem.

Since $\{u_{n}\}$ and $\{v_{n}\}$ are bounded in $W^{1,p}(\Omega)$ and $W^{1,q}(\Omega)$ respectively, there exist subsequences $\{u_{n_{j}}\}\subset\{u_{n}\}$ and $\{v_{n_{j}}\}\subset\{v_{n}\}$  such that
\begin{equation*}
	u_{n_{j}}\rightharpoonup\ui\;\;\textrm{weakly in}\;\;W^{1,p}(\Omega)
	\quad \mbox{and}\quad  v_{n_{j}}\rightharpoonup\vi\;\;\textrm{weakly in}\;\;W^{1,q}(\Omega),
\end{equation*}
as $ n_{j}\to\infty$.

By the Rellich-Kondrachov theorem, there exists a subsequence $\{u_{n_j}\}$ that converges strongly in $L^p(\Omega)$ and a subsequence $\{v_{n_j}\}$
that converges strongly in $L^q(\Omega)$. Extracting again a subsequence if needed, we get
\begin{equation*}
	u_{n_{j}}\longrightarrow\ui\quad \textrm{and}\quad v_{n_{j}}\longrightarrow\vi\;\;
	\textrm{a.e. } \Omega,
\end{equation*} 
as $ n_{j}\to\infty$.
Using that $\{u_n\}$ and $\{v_n\}$ are increasing we obtain that the entire sequences converge pointwise (and weakly) to the unique limits
(that are given by the supremums of the sequences). 
This implies that $\ui\geq\vi$ a.e. $\Omega$, $\ui=f$ and $\vi=g$ a.e. $\partial\Omega$ in the sense of traces. This gives us $\ui\in\tl^{p}_{f,\vi}$ and $\vi\in\tu^{q}_{g,\ui}$.

Let us define $u=\ol(\lpp,\vi,f)$ and $v=\ou(\lqq,\ui,g)$. Our next step is to prove
that $u=\ui$ and $v=\vi$ a.e. $\Omega$ and to get that the convergence is strong in 
the corresponding Sobolev spaces.

First, we begin with the case of the upper membranes $u_{n}$. Since $u=\ol(\lpp,\vi,f)$ and $\ui\in \tl^{p}_{f,\vi}$ we have that $E_{p}(u)\leq E_{p}(\ui)$. On the other hand, since $\{v_{n}\}$ is an increasing sequence we have $\ui\geq\vi\geq v_{n}$ for each $n\in\mathbb{N}$ wich implies $u\in\tl^{p}_{f,v_{n}}$. Then, we have that $E_{p}(u_{n})\leq E_{p}(u)$.

On the other hand, by the semicontinuity of the norm, $u_{n}=\ui$ in $\partial\Omega$ and by weak convergence, we obtain 
\begin{equation*}
	\|\nabla \ui\|_{L^{p}(\Omega)}\leq \liminf\limits_{n\to\infty}\|\nabla u_{n}\|_{L^{p}(\Omega)},\;\; \iio h_{p}u_{n}\longrightarrow\iio h_{p}\ui\;\;\textrm{as}\;\;n\to\infty.
\end{equation*} Then, we have
\begin{equation}
\label{2}	E_{p}(\ui)\leq \liminf\limits_{n\to\infty}E_{p}(u_{n})\leq \limsup\limits_{n\to\infty}E_{p}(u_{n})\leq E_{p}(u)\leq E_{p}(\ui).
\end{equation}
Thus, $E_{p}(\ui)=E_{p}(u)$. And since the energy minimizer is unique, $\ui=u$ a.e. $\Omega$. Moreover, by \eqref{2} and the weak convergence, we have obtained that $\{u_{n}\}$ is a minimizing sequence of the energy and the gradient of $ u_{n}$ converges strongly to the gradient of $\ui$ as $n$ goes to infinity. Then, we conclude that
\begin{equation*}
	u_{n}\longrightarrow\ui\;\;\textrm{strongly in}\;\;W^{1,p}(\Omega)\;\;\textrm{as}\;\;n\to\infty.
\end{equation*}

Once we have seen that $\ui=\ol(\lpp,\vi,f)$, let us prove that $\vi=\ou(\lqq,\ui,g)$.
Again, since $\vi\in\tu^{q}_{g,\ui}$, we have that $E_{p}(v)\leq E_{p}(\vi)$.

Note that we cannot repeat the previous argument step by sept because we do not know if  $v$ is in $\tu^{q}_{g,u_{n-1}}$. Though, it  is enough to change a little our strategy
(we use at this point that  $p\geq q$). We construct a subsequence $\{\tilde{v}_{n}\}$ such that $\tilde{v}_{n}\in \tu^{q}_{g,u_{n-1}}$ and $E_{p}(\tilde{v}_{n})$ converges to $E_{p}(v)$ as $n$ goes to infinity.
We define $\{\tilde{v}_{n}\}$ as $$\tilde{v}_{n}=v-\ui+u_{n}.$$ 
Since $p\geq q$ we have that  $\tilde{v}_{n}\in W^{1,q}(\Omega)$, $\tilde{v}_{n}-g\in W^{1,q}_0(\Omega)$ and $E_{q}(\tilde{v}_{n})$ goes to $E_{q}(v)$ as $n\to\infty$ due to the Rellick-Kondrachov Compactness Theorem and the fact that $u_{n}$ converges to $\ui$ in $W^{1,p}(\Omega)$. Also we have the inequality $\tilde{v}_{n}\leq u_{n}$ because $v\leq \ui$. Then, $\tilde{v}_{n}\in \tu^{q}_{g,u_{n}}$ which implies $E_{q}( v_{n+1})\leq E_{q}(\tilde{v}_{n})$.
Besides, by the semicontinuity of the norm, the fact that $v_{n}=\vi$ on $\partial\Omega$ and the weak convergence, we obtain
\begin{equation*}
	\|\nabla \vi\|_{L^{q}(\Omega)}\leq \liminf\limits_{n\to\infty}\|\nabla v_{n+1}\|_{L^{q}(\Omega)},\;\;\iio h_{q}v_{n+1}\longrightarrow\iio h_{q}\vi\;\;\textrm{as}\;\;n\to\infty.
\end{equation*} As a consequence,
\begin{equation*}
\begin{array}{l}
\displaystyle 
	E_{q}(\vi)\leq \liminf\limits_{n\to\infty}E_{q}(v_{n+1})\leq\limsup\limits_{n\to\infty}E_{q}(v_{n+1}) \\[10pt]
	\qquad \displaystyle 
	\leq \limsup\limits_{n\to\infty}E_{q}(\tilde{v}_{n})= E_{q}(v)\leq E_{p}(\vi).
	\end{array}
\end{equation*}
Then, $E_{q}(\vi)=E_{q}(v)$. By the uniqueness of the minimizer, $\vi=v$ a.e. $\Omega$. Furthermore, from the same reasons as in the case of the upper membranes, we have that the entire sequence converges strongly,
\begin{equation*}
		v_{n}\longrightarrow\vi\;\;\textrm{strongly in}\;\;W^{1,q}(\Omega)\;\;\textrm{as}\;\;n\to\infty.
\end{equation*}
This ends the proof.
\end{proof}

If $p<q$ we can also construct a pair of sequences that converges in $W^{1,p}(\Omega)$ and $W^{1,q}(\Omega)$ to a solution of the two membranes problem.
In this case we just have to start the iterations of the obstacle problems with $u_{0}$ a weak  supersolution to $ \lpp u=0$.

\begin{theorem}
	For $p<q$ let 
	$f\in W^{1,p}(\Omega)$ and $g\in W^{1,q}(\Omega)$ be two functions such that $f\geq g$ in $\partial\Omega$ in the sense of traces and take $h_p\in W^{-1,p}(\Omega)$ and $h_{q}\in W^{-1,q}(\Omega)$. 
	Take $u_{0}$ a  weak supersolution to $ \lpp u=0$ in $\Omega$ such that \mbox{$u_{0}-f\in W^{1,p}_{0}(\Omega)$} and then let
	\begin{equation*}
		u_{n}=\ol(\lpp,v_{n-1},f ),\qquad
		v_{n}=\ou( \lqq,u_{n},g).
	\end{equation*}

	Both sequences $\{u_{n}\}_{n=0}^{\infty}\subset W^{1,p}(\Omega)$ and $\{v_{n}\}_{n=0}^{\infty}\subset W^{1,q}(\Omega)$ converge strongly in $W^{1,p}(\Omega)$ and $W^{1,q}(\Omega)$, respectively. Moreover, the limits of the sequences are a solution of the two membranes problem. That is, there exists a pair of functions  $u_{\infty}\in\tl^{\,p}_{f,\vi}$ and $v_{\infty}\in \tu^{\,q}_{g,\ui}$ such that  
	\begin{equation*}
		u_{n}\to u_{\infty}\;\;\textrm{strongly in} \;\;W^{1,p}(\Omega)\;\;\textrm{and}\;\;v_{n}\to v_{\infty}\;\;\textrm{strongly in}\;\; W^{1,q}(\Omega),
	\end{equation*}
	and, in addition, the limit pair $(u_{\infty},v_{\infty})$ satisfies
	\begin{equation*}
		u_{\infty}=\ol(\lpp,\vi,f )\;\;\textrm{and}\;\;
		v_{\infty}=\ou(\lqq,\ui,g).
	\end{equation*}
\end{theorem}

\begin{proof}
	The difference between these sequences and the sequences defined in proposition \eqref{prop2} is that, here, $\{u_{n}\}$ and $\{v_{n}\}$ are decreasing sequences.
	
	Due to the monotonicity of the sequences, we can prove that $u_{n}$ converges weakly to some $\ui$ in $W^{1,p}(\Omega)$ and $v_{n}$ converges to some $\vi$ weakly in $W^{1,q}(\Omega)$ as $n$ goes to infinity.
	
	The strong convergence of $v_{n}$ to $\vi$ is given in the same way that we got strong convergence for $u_{n}$ to $\ui$ in proposition \eqref{prop2} thanks to $\{v_{n}\}$ is a decreasing sequence. On the other hand, for $\{u_{n}\}$ we can reproduce the proof for $\{v_{n}\}$ in proposition \eqref{prop2} because  $W^{1,q}(\Omega)\hookrightarrow W^{1,p}(\Omega)$ continuously.
\end{proof}

\begin{remark} An alternative idea to deal with the case 
$p<q$ runs as follows:
It can be easily proved that 
	\begin{equation*}
		u=\ol(-\Delta_{p}+h_{p},\varphi, f)\;\;\textrm{if and only if}\;\;u=\ou(-\Delta_{p}-h_{p},-\varphi, -f).
	\end{equation*} 
Then, if $p<q$ we consider 
\begin{align*}
	&p'=q,\;\;\; h_{p'}=-h_{q},\;\;\; f'=-g,\\&q'=p,\;\;\; h_{q'}=-h_{p},\;\;\; g'=-f.
\end{align*}
For the problem with the new set of parameters,
$p'$, $q'$, $h_{p'}$, $h_{q'}$, $f'$ and $g'$, 
we can apply the iterative method described in Theorem \eqref{prop2} and get a pair $(\ui',\vi')$ such that $\ui\in\tl^{p'}_{f',\vi'}$ and $\vi'\in\tu^{q'}_{g',\ui'}$ such that $\ui'=\ol(-\Delta_{p'}+h_{p'},\vi',f')$ and $\vi'=\ou(-\Delta_{q'}+h_{q'},\ui',g')$. Thus, $\ui=-\vi'$ and $\vi=-\ui'$ are in $\tl^{p}_{f,\vi}$ and $\tu^{q}_{g,\ui}$ respectively and
\begin{equation*}
	\ui=\ol( \lpp,\vi,f), \;\;\textrm{and}\;\; \vi=\ou( \lqq,\ui,g),
\end{equation*}
i.e, $(\ui,\vi)$ is a solution for the original two membranes problem.
\end{remark}

\begin{remark}  \label{rem4}
We remark that the limit depends strongly on the initial function
from where we start the iterations. We would like to highlight that $\{u_{n}\}$ and $\{v_{n}\}$ converge to $u_{\infty}$ and $v_{\infty}$ strongly in $W^{1,p}(\Omega)$ and $W^{1,q}(\Omega)$ respectively due to the fact that both sequences are monotone. The monotonicity, in turn, depends on the initial function for the iteration. Specifically, when $p\geq q$, the monotonicity arises because the initial datum $v_{0}$ is a weak subsolution of $\mathcal{L}_{q}v=0$. On the other hand, when $p\leq q$ it stems from the fact that $u_{0}$ is a weak supersolution of $\mathcal{L}_{p}u=0$.

If one considers the pair $(u,v)$ obtained as follows:
let $u$ be the solution to 
$ \lpp u=0$ in $\Omega$ with $u=f$ on $\partial \Omega$ and $v$ 
the solution to the corresponding obstacle problem 
$v=\ou( \lqq,u,g)$, then, as we have mentioned in the introduction
this pair $(u,v)$ is a solution to the two membranes problem. Now,
if one starts the iteration procedure with $u_0 =u$ we obtain that the sequences
converge after only one step. In fact we get
 $v_1 =\ou( \lqq,u,g)$
and next the iteration gives again, that is, we have $u =\ol( \lpp,v_1,f)$. 
This is due to the fact that $u$ is a solution to
$ \lpp u=0$ in the whole $\Omega$ with $u\geq v_1$ and therefore it is
the solution to the obstacle problem (with $v_1$ as obstacle from below).  
Analogously, if one starts with $v$ the solution to 
$ \lqq v = 0$ in $\Omega$ with $v=g$ on $\partial \Omega$ we obtain $u$ 
as the solution to the corresponding obstacle problem 
$u=\ol( \lpp,v,f)$. This pair $(u,v)$ is also a solution to the 
 two membranes problem, but, in general the two pairs are different. 
\end{remark}

\subsection{Extension to the two membranes problem for nonlocal operators }
If we consider 
$$
	\lpp:=(-\Delta_{p})^{s}+h_{p}\quad \mbox{ and } \quad
	\lqq:=(-\Delta_{q})^{t}+h_{q}
$$
with
$$
(-\Delta_{p})^{s} (u) (x) = \int_{\mathbb{R}^N} 
\frac{|u(y) -u(x)|^{p-2} (u(y) -u(x))}{|x-y|^{N+sp}} dy,
$$
the fractional $p-$Laplacian. Notice that $\lpp$ and 
$\lqq$
 are associated with the energies
 $$
E^s_{p}(w)=\frac{1}{p}
 \int_{\mathbb{R}^N}\int_{\mathbb{R}^N} 
\frac{|u(y) -u(x)|^{p}}{|x-y|^{N+sp}} dy \, dx+\int_{\Omega}h_{p}w 
$$
and
$$
 E^t_{q}(w)=\frac{1}{q} \int_{\mathbb{R}^N}\int_{\mathbb{R}^N} 
\frac{|u(y) -u(x)|^{q}}{|x-y|^{N+tq}} dy \, dx+\int_{\Omega}h_{q}w.
$$

Our iterative method of the Theorem \eqref{prop2} gives us a pair $(\ui,\vi)$ that is solution of the ``new'' two membranes problem  whenever
\begin{equation}
	\label{01} W^{s,p}(\Omega)\hookrightarrow W^{t,q}(\Omega)\;\;\textrm{continuosly.}
\end{equation}

	To be able to define a variational solution of the obstacle problem for these new $\lpp$ and $\lqq$ is nedeed that $\tl^{p}_{f,\varphi}$ and  $\tu^{q}_{g,\varphi}$ are closed sets under the weak topology. This is given by Mazur's Theorem if $W^{s,p}(\Omega)$ and $W^{t,q}(\Omega)$ are reflexive Banach space and $\tl^{p}_{f,\varphi}$ and $\tu^{q}_{g,\varphi}$ are closed convex sets (with the norm) in $W^{s,p}(\Omega)$ and $W^{t,q}(\Omega)$ respectively.


\section{Viscosity solutions}
\label{sect3}

In this section we aim to prove similar convergence results for iterations of the obstacle
problem to a pair that solves the two membranes problems when the involved
operators are not variational. Here we will
consider two nonlinear elliptic operators and we understand solutions in the viscosity sense.
To define the sequences as solutions to the corresponding obstacle problems we
need to assume that the involved operators verify the following set of conditions: 

{\bf {Hypothesis:}}
	The operators $\lu$ and $\ld$ verify
		\begin{itemize}
			\item $\li w= F(D^{2}w,\nabla w)-h(x)$ with $F$ continuous in both coordinates, such that the comparison principle holds.   
		\item If  $f$ is continuous, there exists a unique $w\in C(\overline{\Omega})$ solution to $\li w=0$ in $\Omega$, $w=f$ in $\partial\Omega$.
		\item If $\varphi$, $\psi$ and $f$, $g$ are continuous, 
		then the corresponding solutions to the obstacle problems, $u=\ool(\lu,\varphi,f)$ and $v=\oou(\ld,\psi,g)$ are continuous. See Definition \ref{def3.4} below.
	\end{itemize}


To be precise, working in the viscosity sense, we need to introduce the definition of
semicontinuous functions.  

\begin{definition}
	$f:\Omega\longrightarrow\RR$ is a lower semicontinuous function, l.s.c,  at $x\in\Omega$ if for each $\ep>0$ there exists a $\delta>0$ such that 
	\begin{equation*}
		f(x)\leq f(y)+\ep\;\;\textrm{for all}\;\;y\in B_{\delta}(x).
	\end{equation*}
	If a function $f$ is l.s.c. at every $x\in\Omega$, we say that $f$ is l.s.c. in $\Omega.$
	The lower semicontinuous envelope of $f$ is
	\begin{equation*}
		f_{*}=\sup\{h:\Omega\longrightarrow\RR\;\;\textrm{l.s.c.}\,:\;\;h\leq f.\}
	\end{equation*}
	On the other hand, 
	$g:\Omega\longrightarrow\RR$ is a upper semicontinuous function, u.s.c,  at $x\in\Omega$ if for each $\ep>0$ there exists a $\delta>0$ such that 
	\begin{equation*}
		g(y)\leq g(x)+\ep\;\;\textrm{for all}\;\;y\in B_{\delta}(x).
	\end{equation*}
	If a functions $g$ is u.s.c. at every $x\in\Omega$, we say that it is u.s.c. in $\Omega.$
	The upper semicontinuous envelope of $g$ is
	\begin{equation*}
		g^{*}=\inf\{h:\Omega\longrightarrow\RR\;\;\textrm{u.s.c.}\,:\;\;h\geq f.\}
	\end{equation*}
\end{definition}

Now, we give the precise definition of sub and supersolutions in 
the viscosity sense.

\begin{definition}
 The function $u:\Omega\longrightarrow \RR$ is called a $\li$-viscosity supersolution if its lower semicontinuous envelope function, $u_{*}$, satisfies the following:
	for every $\phi\in C^{2}(\overline{\Omega})$ such that $\phi$ touches $u_{*}$ at $x\in\Omega$ strictly from below, that is, $u_{*}-\phi$ has a strict minimum at $x$ with $u_{*}(x)=\phi (x)$, we have 
	\begin{equation*}
		\li \phi(x)\geq 0.
	\end{equation*}
	In this case, we write $\li u\geq0$ in the viscosity sense.

	Conversely, $u$ is called a $\li$-viscosity subsolution if its upper semicontinuous envelope function, $u^{*}$, satisfies that for every $\phi\in C^{2}(\overline{\Omega})$ such that $\phi$ touches $u^{*}$ at $x\in\Omega$ strictly from above, that is, $\phi-u^{*}$ has a strict minimum at $x$ with $u^{*}(x)=\phi (x)$, we have 
	\begin{equation*}
		\li \phi(x)\leq 0
	\end{equation*}
	and we write $\li u\leq 0$ in the viscosity sense. 
	
	Finally, $u$ is a $\li$-viscosity solution if it is both a $\li$-viscosity supersolution and a $\li$-viscosity subsolution and we denote $\li u=0$ in the viscosity sense.
\end{definition}

Now, let us introduce three results concerning the limit of non-decreasing sequences of viscosity sub or supersolutions. Although these results are not difficult to prove, they will be crucial for the proof of our main Theorem \eqref{v.maintheorem} in this section.

\begin{proposition}\label{propv2}
	Let be $\{v_{n}\}_{n=0}^{\infty}\subset C(\overline{\Omega})$ a non-decreasing sequence of functions, continuous up to the boundary such that $v_{n}$ is a $\li$-viscosity subsolution for each $n\in\mathbb{N}$ and $v_{n}\longrightarrow \vi$ pointwise. Then, $\vi$ is a $\li$-viscosity subsolution.
\end{proposition}

\begin{proof}
	Let be $\phi\in C^{2}(\overline{\Omega})$ such that $\phi-\vi^{*}$ has a strict minimum at $x_{0}\in\Omega$ and $\phi(x_{0})=\vi^{*}(\xo)$.
	
	Since $\vi^{*}$ is upper semicontinuous, $\phi-\vi^{*}$ is lower semicontinuous. Then, fixed $r>0$ small enough, $\phi-\vi^{*}$ reaches a minimun in $D_{r}:=\overline{\Omega}\setminus B_{r}(\xo)$. Say in $z_{\infty}\in D_{r}$. Thanks to $- \vi^{*}=(-\vi)_{*}$ and that the sequence $\{v_{n}\}$ is non-decreasing, we have
	\begin{equation}
		\label{vi1}0<(\phi-\vi^{*})(z_{\infty})\leq (\phi-\vi^{*})(z)\leq (\phi-\vi)(z)\leq (\phi-v_{n})(z)
	\end{equation}
	for each $z\in D_{r}$ and for each $n\in\mathbb{N}$.
	
	On the other hand, by the definition of pointwise value of the lower semicontinuous envelope,
	\begin{equation*}
	\begin{array}{l}
	\displaystyle
		0=(\phi-\vi^{*})(\xo)=(\phi+(-\vi)_{*})(\xo) \\[10pt]
		\qquad \displaystyle
=\inf_{\Big \{ \{x_{k}\}_{k=0}^{\infty}\subset \Omega: \;\;x_{k}\rightarrow_{k\to\infty} \xo\Big\}}\lim\limits_{k\to\infty}(\phi-\vi)(x_{k}).
\end{array}
	\end{equation*}
	Then, given $\ep>0$ there exists a sequence $\{y_{k}\}$ within $B_{r/2}(\xo)$ such that
	$$\lim_{k\to\infty}|(\phi-\vi)(y_{k})|\leq \ep/3.$$ Thus, there is $k_{0}\in\mathbb{N}$ such that $|(\phi-\vi)(y_{k_{0}})|\leq 2\ep/3$. Since $v_{n}$ converges to $\vi$ pointwise, there exists $n_{0}\in \mathbb{N}$ such that $|(\phi-\vi)(y_{k_{0}})-(\phi-v_{n})(y_{k_{0}})|\leq \ep/3$ for all $n\geq n_{0}$.
	
	In short, for each $\ep>0$ there exists a $y\in B_{r/2}(\xo)$ and $n_{0}\in \mathbb{N}$ such that 
	\begin{equation*}
		(\phi-v_{n})(y)\leq \ep\;\;\textrm{for all}\;\;n\geq n_{0}.
	\end{equation*}
	Since, $\phi-v_{n}$ is a continuous function, given $\ep<(\phi-\vi^{*})(z_{\infty})$ one can see that the minimum of $\phi-v_{n}$ in $\overline{\Omega}$ is reached at some $x_{n}^{r}\in B_{r/2}(\xo)$ for all $n\geq n_{0}$ thanks to the above inequality and \eqref{vi1}. Note that $n_{0}$ depends on $r$, so we write $n_{0}=n_{0}(r)$. We have $n_{0}(r)\to \infty$ and $x_{n}^{r}\to \xo$ as $r\to0^{+}$. 
	Due to the fact that $v_{n}$ is a $\li$- viscosity subsolution, we have
	\begin{equation*}
		\li\phi(x_{n}^{r})\leq 0.
	\end{equation*}

	Since $\li$ is a continuous operator we have
	\begin{equation*}
		0\geq \lim\limits_{r\to\infty}(\li\phi)(x_{n}^{r})=(\li \phi)\big(\lim\limits_{r\to\infty} x_{n}^{r}\big)=(\li \phi)(\xo).
	\end{equation*}
	Therefore $\vi$ is a $\li$-viscosity subsolution.
\end{proof}

\begin{lemma}\label{remark3}
	Let $\{u_{n}\}\subset C(\overline{\Omega})$ a non-decreasing sequence of continuous up to the boundary functions. If $u_{n}$ converge pointwise to some function $\ui$, then $\ui$ is a lower semicontinuous function  in $\overline{\Omega}$.
\end{lemma}

\begin{proof}
	Let be $\xo\in\Omega$. Since $u_{n}\longrightarrow\ui$ pointwise, we have that for each $\ep>0$ there exists a $n_{0}\in\mathbb{N}$ such that $\ui(\xo)\leq u_{n_{0}}(\xo)+\ep.$ On the other hand, due to the fact that $u_{n_{0}}$ is continuous there exist a $\delta$ such that $u_{n_{0}}(\xo)\leq u_{n_{0}}(y)+\ep$ for all $y\in B_{\delta}(\xo)$. Moreover, thanks to the fact that the sequence $\{u_{n}\}$ is non-decreasing we get that, for every $\ep>0$ there exists $\delta>0$ such that 
	\begin{equation*}
		\ui(\xo)\leq \ui(y)+\ep\;\;\textrm{for all}\;\; y\in B_{\delta}(\xo).
	\end{equation*}
	The proof is finished.
\end{proof}

\begin{proposition}\label{propv3}
	Let $\{u_{n}\}_{n=0}^{\infty}\subset C(\overline{\Omega})$ be a non-decreasing sequence of continuous up to the boundary functions such that $u_{n}$ is a $\li$-viscosity supersolution for each $n\in\mathbb{N}$ and $u_{n}\longrightarrow \ui$ pointwise. Then, $\ui$ is a $\li$-viscosity supersolution. 
\end{proposition}

\begin{proof}
	By Lemma \ref{remark3} we know that $\ui$ is a lower semicontinuous function. To prove that $\ui$ is a $\li$-viscosity supersolution let us consider $\phi\in C^{2}(\overline{\Omega})$ such that $\ui-\phi$ has a strict minimum at $\xo\in\Omega$ such that $(\uip)(\xo)=0$. We want to obtain that $(\li\phi)(\xo)\geq 0$.
	
	Let $r>0$ be a fixed small radius. Since $\uip$ is lower semicontinuous, these functions reach a positive minimum in $\overline{\Omega}\setminus B_{r}(\xo)$, say in $\zi\in \overline{\Omega}\setminus B_{r}(\xo)$, then, we have
	\begin{equation}
		0<(\uip)(\zi)\leq (\uip)(x)\;\;\textrm{for all}\;\;x\in\overline{\Omega}\setminus B_{r}(\xo).
	\end{equation}
	Also, since $\unp$ is continuous, there exists $z_{n}\in\overline{\Omega}\setminus B_{r}(\xo)$ where $\unp$ reaches a minimum in $\overline{\Omega}\setminus B_{r}(\xo)$. We claim that $(\unp)(z_{n})$ converges to $(\uip)(\zi)$ as $n$ to infinity. Let us continue
	the proof assuming this fact and prove it after finishing our argument.
	
	The claim implies that there exists $n_{0}=n_{0}(r)$ such that 
	\begin{equation}
		\label{vi2} 0<\frac{1}{2}(\uip)(\zi)\leq (\unp)(z_{n})\leq (\unp)(x)
	\end{equation}
	for all $x\in\overline{\Omega}\setminus B_{r}(\xo)$
	and $n\geq n_{0}$.
	
	On the other hand, since $u_{n}$ converges to $\ui$ pointwise and \mbox{$(\uip)(\xo)=0$,} there exists $n_{1}=n_{1}(r)$ such that 
	\begin{equation}
		(\unp)(\xo)\leq\frac{1}{2}(\uip)(\zi)\;\;\textrm{for all}\;n\geq n_{1}.
	\end{equation}
	Thus, the above and \eqref{vi2} imply $\unp$ reaches its minimum in $\overline{\Omega}$ in the interior of  the ball $B_{r}(\xo)$ for all $n\geq n_{3}=\max\{n_{0},n_{1}\}$. Call this point $x_{n}^{r}$. Then, since $u_{n}$ is a $\li$-viscosity supersolution, we get
	\begin{equation*}
		(\li\phi)(x_{n}^{r})\geq 0\;\;\textrm{for all}\;\;n\geq n_{3}(r)\;\;\textrm{with}\;\;x_{n}^{r}\in B_{r}(\xo).
	\end{equation*}
	Due to the fact that $x_{n}^{r}\to \xo$ as $r\to 0^{+}$, we have
	\begin{equation*}
		(\li\phi)(\xo)=\lim\limits_{r\to0^{+}}(\li\phi)(x_{n}^{r})\geq 0.
	\end{equation*}
	Therefore, $\ui$ is a $\li$-viscosity supersolution in $\Omega$.

	Now, we will prove the claim. Let us define  $w_{\infty}=\uip$, $w_{n}=\unp$ and $a_{n}=w_{n}(z_{n})$. Thanks to $\{u_{n}\}$ is a pointwise increasing sequence, $\{a_{n}\}$ is also an increasing sequence,
	\begin{equation*}
		a_{n}=w_{n}(z_{n})\leq w_{n}(z_{n+1})\leq w_{n+1}(z_{n+1})=a_{n+1}.
	\end{equation*}
	Moreover, the sequence is bounded above by $w_{\infty}(\zi)$ due to the fact that 
	$u_{n}$ converges pointwise to $\ui$. Hence,
	\begin{equation*}
		a_{n}\leq w_{n}(z_{n})\leq w_{n}(\zi)\leq w_{\infty}(\zi).
	\end{equation*}
	Then, there exists a number $a_{\infty}\leq w_{\infty}(\zi)$ such that $a_{n}$ converges to $a_{\infty}$ as $n$ tends to infinity.
	
	Suppose that $a_{\infty}<w_{\infty}(\zi)$. Then, because $\{z_{n}\}$ is a colection of points in a compact set, there exists a subsequence $\{z_{n_{j}}\}\subset\{z_{n}\}$ and a point $\widetilde{z}\in\overline{\Omega}\setminus B_{r}(\xo)$  sucht that 
	\begin{equation*}
		z_{n_{j}}\longrightarrow\widetilde{z}\;\;\textrm{as}\;\;n_{j}\to\infty.
	\end{equation*}
	Since $w_{n}$ is a lower semicontinuous function, for each $\ep>0$, there exists a $\delta=\delta(\ep,\widetilde{z})$ such that
	\begin{equation*}
		w_{\infty}(\widetilde{z})\leq w_{\infty}(y)+\ep\;\;\textrm{for all}\;\;y\in B_{\delta}(\widetilde{z}).
	\end{equation*}
	Due to the fact that $z_{n_{j}}$ converges to $\widetilde{z}$, we can suppose $z_{n_{j}}\in B_{\delta}(\widetilde{z})$ for all $n_{j}$. Then 
	\begin{equation*}
		w_{\infty}(\widetilde{z})\leq w_{\infty}(z_{n_{j}})+\ep\;\;\textrm{for all}\;\;n_{j}.
	\end{equation*}
	For each $n_{j}$  there exists a $m_{j}\in \mathbb{N}$ such that 
	\begin{equation*}
		w_{\infty}(z_{n_{j}})\leq w_{k}(z_{n_{j}})+\ep\;\;\textrm{for all}\;\;k\geq m_{j}.
	\end{equation*}
	Observe that we can construct a sequence $\{(n_{j},k_{j})\}$ where $k_{j}\geq m_{j}$ for all $j\in\mathbb{N}$ and $(n_{j},k_{j})\longrightarrow(\infty,\infty)$ as $j\to\infty$  such that 
	\begin{equation*}
		w_{\infty}(z_{n_{j}})\leq w_{k_{j}}(z_{n_{j}})+\ep\;\;\textrm{for all}\;\;j\in\mathbb{N}.
	\end{equation*}
	
	Since $a_{\infty}<w_{\infty}(\zi)$ there exists $\ep_{0}>0$ such that 
	\begin{equation*}
		a_{\infty}+2\ep_{0}<w_{\infty}(\zi).
	\end{equation*}
	For that $\ep_{0}$ the above argument implies that there exists a sequence $\{(n_{j},k_{j})\}$ that goes to $(\infty,\infty)$ as $j\to\infty$ such that 
	\begin{equation*}
		a_{\infty}+2\ep_{0}<w_{\infty}(\zi)\leq  w_{k_{j}}(z_{n_{j}})+2\ep_{0}\;\;\textrm{for all}\;\;j\in\mathbb{N}
	\end{equation*}
	and therefore,
	\begin{equation*}
		a_{\infty}< w_{k_{j}}(z_{n_{j}})\;\;\textrm{for all}\;\;j\in\mathbb{N}.
	\end{equation*}
	On the other hand, since $\{a_{n}\}$ is a increasing sequence and $a_{k_{j}}=w_{k_{j}}(z_{k_{j}})$ we have that
	\begin{equation*}
		w_{k_{j}}(z_{k_{j}})\leq a_{\infty}< w_{k_{j}}(z_{n_{j}})\;\;\textrm{for all}\;\;j\in\mathbb{N}.
	\end{equation*}
	Thanks to the fact that for each $j\in\mathbb{N}$ the functions $w_{k_{j}}$ are continuous, $\overline{\Omega}\setminus B_{r}(\xo)$ is connected  and $z_{n_{j}}$ and $z_{k_{j}}$ are in $\overline{\Omega}\setminus B_{r}(\xo)$, there exists $y_{j}\in\overline{\Omega}\setminus B_{r}(\xo)$ such that 
	\begin{equation*}
		w_{k_{j}}(y_{j})=a_{\infty}.
	\end{equation*}
	Thus, there is a subsequence $\{y_{j_{i}}\}\subset\{y_{j}\}$ and a point $y\in \overline{\Omega}\setminus B_{r}(\xo)$ such that $y_{j_{i}}\to y$ as $i\to\infty$. Due to the fact that $w_{\infty}$ is a lower semicontinuous function we have
	\begin{equation*}
		w_{\infty}(y)\leq \lim\limits_{y_{j_{i}}\to y}w_{k_{j_{i}}}(y_{j_{i}})=a_{\infty}.
	\end{equation*}
	However, this is a contradiction because $y\in\overline{\Omega}\setminus B_{r}(\xo)$ and that implies
	\begin{equation*}
		a_{\infty}<w_{\infty}(\zi)\leq w_{\infty}(y).
	\end{equation*}
	Therefore, we have the desired equality $a_{\infty}=w_{\infty}(\zi)$.
	\end{proof}
	
	Now, we are ready to introduce the definition of a solution to the obstacle problem 
	in the viscosity sense.

\begin{definition} \label{def3.4}
	Given $\varphi:\overline{\Omega}\rightarrow\RR$ and $f\in C(\partial\Omega)$ such that $f\geq\varphi^{*}$ in $\partial\Omega$.
	We say that $u:\Omega\longrightarrow\RR$ is a viscosity solution of the $\li$-lower obstacle problem with obstacle $\varphi$ and boundary datum $f$ and we denote $u=\ool(\li,\varphi,f)$ if $u$ satisfies:
	\begin{equation*}
		\begin{cases}
			u_{*}\geq \varphi^{*}&\;\;\textrm{in}\;\;\Omega,\\
			u_{*}=f&\;\;\textrm{in}\;\;\partial\Omega,\\
			\li u\geq 0&\;\;\textrm{in}\;\;\Omega\;\;\textrm{in the viscosity sense,}\\
			\li u=0&\;\;\textrm{in}\;\;\{ u_{*}>\varphi^{*}\}\;\;\textrm{in  the viscosity sense.}
		\end{cases}
	\end{equation*}
	
	Whereas, given $\psi$ and $g$ such that $f\leq\psi_{*}$ on $\partial\Omega$, we say that $v$ is a viscosity solution of the $\li$-upper obstacle problem with obstacle $\psi$ and boundary datum $g$, and we denote $v=\oou(\li,\psi,g)$, if $v$ satisfies:
	\begin{equation*}
		\begin{cases}
			v^{*}\leq \psi_{*}&\;\;\textrm{in}\;\;\Omega,\\
			v^{*}=g&\;\;\textrm{in}\;\;\partial\Omega,\\
			\li v\leq 0&\;\;\textrm{in}\;\;\Omega\;\;\textrm{in the viscosity sense,}\\
			\li v=0&\;\;\textrm{in}\;\;\{v^{*}<\psi_{*}\}\;\;\textrm{in the viscosity sense.}
		\end{cases}
	\end{equation*}
\end{definition}

\begin{remark}
	Note that, when $u$ is a viscosity solution of the $\li$-lower obstacle problem with obstacle $\varphi$ and boundary datum $f$, then, its lower and upper semicontinuous envelopes, $u_{*}$ and $u^{*}$, are also viscosity solutions for the same problem. 
	
	The same applies to the upper obstacle, if $v=\oou(\li,\psi,g)$, then, $v_{*}=\oou(\li,\psi,g)$ and $v^{*}=\oou(\li,\psi,g)$.

\end{remark}

However, we have uniqueness of solutions to the obstacle problem 
up to semicontinuous envelopes.

\begin{lemma}
	\label{unique}
	Given $\varphi$ and $f$ defined as before
	\begin{itemize}
		\item[(a)] if $f\geq\varphi^{*}$ in $\partial\Omega$, there exists at most one lower semicontinuous function $u$ such that $u=\ool(\li,\varphi,f)$.
		
		\medskip
		
		\item[(b)] 
		if $f\leq\varphi^{*}$ in $\partial\Omega$, there exists at most one upper semicontinuous function $u$ such that $u=\oou(\li,\varphi,f)$.
	\end{itemize}
\end{lemma}

\begin{proof} We will only prove item (a), the other item is analogous. 

	(a) Suppose that there exists two lower semicontinuous functions solutions $u_1$ and $u_2$, let us prove that $u_2\geq u_1$. 
	
	In the set $\{u_1=\varphi^{*}\}$ we have $u_2\geq \varphi^{*}=u_1$.
	In the open set $\{u_1>\varphi^{*}\}$ we have that $\li u_2\geq 0$ and $\li u_1=0$ and $u_2\geq u_1$ in $\partial\{u_1>\varphi^{*}\}$. Then, by the Comparison Principle we obtain that $u_2\geq u_1$ in $\{u_1>\varphi^{*}\}$.
	Thus, we conclude that $$u_2\geq u_1\;\;\textrm{in}\;\;\Omega.$$
	
	The reverse inequality follows interchanging the roles of $u_1$ and $u_2$.
\end{proof}

In view of the previous lemma, from now we will suppose that $u=\ool(\li,\varphi,f)$ is lower semicontinuous and $v=\oou(\li,\varphi,f)$ is upper semicontinuous. 

Next, we show that when the obstacles and the boundary data are ordered then the
solutions to the obstacle problems are also ordered.

\begin{lemma}
	\label{comparison} \
	\begin{itemize}
		\item[(a)] Given $\varphi_1$ and $\varphi_2$ functions such that $\varphi_1\leq\varphi_2$ (a.e.) and $f_1, f_2\in C(\partial\Omega)$ such that $f_1\leq f_2$, if $f_1\geq\varphi_1^{*}$ and $f_2\geq\varphi_2^{*}$ on $\partial\Omega$, let us consider  $u_1=\ool(\li,\varphi_1,f)_1$ and $u_2=\ool(\li,\varphi_2,f_2)$, then $$u_1\leq u_2\;\;\textrm{in}\;\;\Omega.$$
		\item[(b)] Given $\psi_1$ and $\psi_2$ functions such that $\psi_1\leq\psi_2$ (a.e.) and $g_1, g_2\in C(\partial\Omega)$ such that $g_1\leq g_2$, if $g_1\leq\psi_{1{*}}$  and $g_2\leq\psi_{2{*}}$ on $\partial\Omega$, let us consider  $v_1=\oou(\li,\psi_1,g_1)$ and $v_2=\oou(\li,\psi_2,g_2)$, then $$v_1\leq v_2 \;\;\textrm{in}\;\;\Omega.$$
	\end{itemize}
\end{lemma}
\begin{proof} We will prove item (a), the other case is analogous. 

	(a) We will consider two cases: 
	
	In the set $\{u_1=\varphi_1^{*}\}$ we have $u_2\geq \varphi_2^{*}\geq\varphi_1^{*}=u_1$.
	
	 In the open set $\{u_1>\varphi_1^{*}\}$ we have that $\li u_2\geq 0$ and $\li u_1=0$ and $u_2\geq u_1$ in $\partial\{u_1>\varphi_1^{*}\}$. Then, by the Comparison Principle we obtain that $u_2\geq u_1$ in $\{u_1>\varphi^{*}\}$. 
	 
	 Thus $u_2\geq u_1$ in $\Omega$.
\end{proof}
\begin{lemma}\label{propv1}
	Given $\varphi$ and $f$ as before, let us consider $u=\ool(\li,\varphi,f)$, then
	\begin{equation*}
	\begin{array}{l}
	\displaystyle 
		u=\min\Big\{w:\Omega\longrightarrow\RR\;\;\textrm{lower semicontinuous}:
		\\[10pt] \qquad \qquad \qquad \displaystyle\begin{aligned}
			&w\geq \varphi^{*}\;\;\textrm{in}\;\;\Omega,\;\;w\geq f\;\;\textrm{on}\;\;\partial\Omega,\\
			&\li w\geq 0\;\;\textrm{in}\;\;\Omega\;\;\;\textrm{in  the viscosity sense}
		\end{aligned}
		\Big \}.
		\end{array}
	\end{equation*}
	
	In the same way, if we take $v=\oou(\li,\varphi,f)$, then
	\begin{equation*}
	\begin{array}{l} \displaystyle
		v=\max\Big\{w:\Omega\longrightarrow\RR\;\;\textrm{upper semicontinuous}:\;\;
		\\[10pt] \qquad \qquad \qquad \displaystyle
		\begin{aligned}
			&w\leq \varphi_{*}\;\;\;\textrm{in}\;\;\Omega,\;\;w\leq f\;\;\textrm{on}\;\;\partial\Omega,\\
			&\li w\leq 0\;\;\textrm{in}\;\;\Omega\;\;\;\textrm{in the viscosity sense}
		\end{aligned}
		\Big \}.
		\end{array}
	\end{equation*}
\end{lemma}

\begin{proof}
	Let us prove the first claim. Consider
	\begin{equation*}
	\begin{array}{l}
	\displaystyle 
		\overline{u}=\min\Big\{w:\Omega\longrightarrow\RR\;\;\textrm{lower semicontinuous}:
		\\[10pt] \qquad \qquad \qquad \displaystyle
		\begin{aligned}
			&w\geq \varphi^{*}\;\;\textrm{in}\;\;\Omega,\;\;w\geq f\;\;\textrm{on}\;\;\partial\Omega,\\
			&\li w\geq 0\;\;\textrm{in}\;\;\Omega\;\;\;\textrm{in the viscosity sense}
		\end{aligned}
		\Big \}.
		\end{array}
	\end{equation*}
	Since $u\geq\varphi^{*}$ in $\Omega$, $u=f$ on $\partial\Omega$ and $\li u\geq 0$ in $\Omega$ in viscosity sense, we get $$\overline{u}\leq u.$$
	
	Now, let us consider 
	$$
	\begin{array}{l}
	\displaystyle 
	w\in\Big\{w:\Omega\longrightarrow\RR\;\;\textrm{lower semicontinuous}:
	\\[10pt] \qquad \qquad \qquad \displaystyle
	\begin{aligned}
		&w\geq \varphi^{*}\;\;\textrm{in}\;\;\Omega,\;\;w\geq f\;\;\textrm{on}\;\;\partial\Omega,\\
		&\li w\geq 0\;\;\textrm{in}\;\;\Omega\;\;\;\textrm{in the viscosity sense}
	\end{aligned}
	\Big \}.
	\end{array}
	$$
	In the set $\{u=\varphi^{*}\}$ we have $w\geq \varphi^{*}=u$.
	
	In the open set $\{u>\varphi^{*}\}$ we have that $\li w\geq 0$ and $\li u=0$ in the viscosity sense and $w\geq u$ in $\partial\{u>\varphi^{*}\}$. Then, by the Comparison Principle we obtain that $w\geq u$ in $\{u_1>\varphi^{*}\}$.
	
	Then, $u\leq w$ in $\Omega$. Taking minimum we get $$u\leq \overline{u}.$$
	
	The other case is analogous.
\end{proof}

Now, we introduce the definition of a solution to the two membranes problem in the 
viscosity sense (this is just Definition \ref{def1.1} with solutions understood 
in the  viscosity sense).


\begin{definition}\label{df2}
	Given $f,g\in C(\partial\Omega)$ with $f\geq g$, let $\lu$ and $\ld$ by two operators that satisfy the hypothesis at the begining of the section. We say a pair of functions $(u,v)$ is a solution of the two membranes problem with boundary data $(f,g)$ if 
	\begin{equation*}
		u=\ool(\lu, v, f) \;\;\textrm{and}\;\;v=\oou(\ld,u,g).
	\end{equation*}
\end{definition}

Our main result is the following.

\begin{theorem} \label{v.maintheorem}
	Given $f,g$ and $\lu,\ld$, let us consider $v_{0}\in C(\overline{\Omega})$ a $\ld$-viscosity subsolution such that $v_{0}\leq g$ in $\partial\Omega$ and 
	then define inductively the sequences
	\begin{equation*}
		u_{n}=\ool(\lu,v_{n},f ),\quad \mbox{ and } \quad 
		v_{n}=\oou( \ld,u_{n-1},g).
	\end{equation*}
	Both  sequences of functions $\{u_{n}\}_{n=0}^{\infty}\subset C(\overline{\Omega})$, $\{v_{n}\}_{n=0}^{\infty}\subset C(\overline{\Omega})$ converge to some limit functions $\ui$ and $\vi$, respectively. Moreover, theses limits are a solution of the two membranes problem with boundary data $(f,g)$, that is,  
	\begin{equation*}
		u_{\infty}=\ool(\lu,\vi,f ) \;\;\textrm{ and }\;\;
		v_{\infty}=\oou(\ld,\ui,g).
	\end{equation*}

In addition, the functions $u_{\infty}$ and $v_{\infty}$ are lower semicontinuous in $\bar{\Omega}$ and continuous in the interior of the set where $u_{\infty}$ touches $v_{\infty}$, i.e, in the interior of $\{ u_{\infty}=v_{\infty}\}$.
\end{theorem}

\begin{proof} As before, we divide the proof in several steps.

	\textbf{First step:} Let us start proving that $\{u_{n}\}_{n=0}^{\infty}$ and $\{v_{n}\}_{n=0}^{\infty}$ are non-decreasing.
	
	Let us see that $v_0\leq v_1$. Recall that $v_1=\oou(\ld,u_0,g)$ and $u_0=\ool(\lu,v_0,f)$, then $v_1\leq u_0$ and $v_0\leq u_0$.
	
	In the set $\{v_1=u_0\}$ we have $v_0\leq u_0=v_1$.
	
	In the open set $\{v_1<u_0\}$ we have that $\li_2 v_1= 0$ and $\li_2 v_0 \leq 0$ and $v_1\geq v_0$ in $\partial\{v_1>u_0\}$. Then, by the Comparison Principle we obtain that $v_1\geq v_0$ in $\{v_1>u_0\}$.
	
	Thus, $v_0\leq v_1$ in $\Omega$.
	
	Now, using Lemma \ref{comparison} we have $u_0\leq u_1$. By induction we can continue and obtain $u_n\leq u_{n+1}$ and $v_n\leq v_{n+1}$ for all $n\geq 1$.
	
	\textbf{Second step:} Let us prove that $\{u_{n}\}_{n=0}^{\infty}$ and $\{v_{n}\}_{n=0}^{\infty}$ are bounded.
	
	Let $w\in C(\overline{\Omega})$ be a $ \li_2$-solution with $w=g$ in $\partial\Omega$. Then, since $\li_2 w= 0$, we get $\li_2 v_n \leq 0$. Moreover, we have $w\geq v_n$ in $\partial\Omega$. Hence, applying the Comparison Principle, we get $w\geq v_{n}$ in $\bar{\Omega}$ for all $n\geq 0$. This implies, together with the fact that $v_{n}\geq v_{0}$, that $\{v_n\}$ is bounded.
	
	 Now, if we consider $z=\ool(\lu,w,f )$, using that $v_n\leq w$ by Lemma \ref{comparison} $u_n\leq z$ and $z\in C(\overline{\Omega})$. This implies that $\{u_n\}$ is bounded. 
	
	\textbf{Third step:} Let us define 
	\begin{equation*}
		\lim_{n\rightarrow\infty}u_n(x)=\ui(x) \quad \quad \mbox{and} \quad \quad \lim_{n\rightarrow\infty}v_n(x)=\vi(x)
	\end{equation*}
	for all $x\in\overline{\Omega}$. Using that $\{u_n\}$ are continuous and the sequence is increasing we obtain that $\ui$ is lower semincontinuous. We also have $\ui=f$ and $\vi=g$ in $\partial\Omega$. Let us consider
	\begin{equation*}
		v=\oou(\ld,\ui,g).
	\end{equation*} 
	This function $v$ is well defined and upper semicontinuous. Our goal is to prove that $v=\vi$.
	
	Using that $u_n\leq \ui$, we get $v_n\leq v$ for all $n\geq 0$ by Lemma  \ref{comparison}. Then, we have that $\vi\leq v$ and $\vi^*\leq v^*=v$. This implies that $\vi^* \leq \ui$. On the other hand, by Proposition \ref{propv2}, we have that $\li_2 \vi^*\leq 0$ in $\Omega$. We also have $\vi^*\geq g$ in $\partial\Omega$. Now, we consider $w\in C(\Omega)$ a solution to
	\begin{equation*}
		\begin{cases}
			\li_2 w= 0&\;\;\textrm{in}\;\;\Omega\;\;\textrm{in the viscosity sense,}\\
			w=g&\;\;\textrm{in}\;\;\partial\Omega.
		\end{cases}
	\end{equation*}
	By the Comparison Principle we have that $v_n\leq w$ in $\overline{\Omega}$. Then $\vi\leq w$ in $\overline{\Omega}$ and $\vi^*\leq w^*=w$ in $\overline{\Omega}$. This implies $v^{*}_{\infty}\leq g$ on $\partial\Omega$ and, therefore, $v^{*}_{\infty}= g$ on $\partial\Omega$. Finally, we will prove that $\li_2\vi^*=0$ in the (open) set $\{\vi^*<\ui\}$. 
	Let us consider $x_0\in\{\vi^*<\ui\}$, and call $\delta=\ui(x_0)-\vi^*(x_0)$. Let us consider $n_0\geq 0$ such that $\ui(x_0)-u_{n_0}(x_0)<\frac{\delta}{4}$. Using that $u_{n_0}$ is continuous, there exists $\eta_1=\eta_1(n_0)>0$ such that $|u_{n_0}(x_0)-u_{n_0}(y)|<\frac{\delta}{4}$ for all $y\in B_{\eta_1}(x_0)$. On the other hand, there exists $\eta_2>0$ such that $\vi^*(y)<\vi^*(x_0)+\frac{\delta}{4}$ for all $y\in B_{\eta_2}(x_0)$. 
	Gathering the previous estimates we obtain
	\begin{equation*}
	\begin{array}{l}
	\displaystyle
		u_{n_0}(y)-\vi^*(y)=\underbrace{u_{n_0}(y)-u_{n_0}(x_0)}_{>-\frac{\delta}{4}}+\underbrace{u_{n_0}(x_0)-\ui(x_0)}_{>-\frac{\delta}{4}} \\[10pt]
		\qquad \qquad		\qquad \qquad \displaystyle+\underbrace{\ui(x_0)-\vi^*(x_0)}_{=\delta}+\underbrace{\vi^*(x_0)-\vi^*(y)}_{>-\frac{\delta}{4}}>\frac{\delta}{4}
		\end{array}
	\end{equation*}
	for all $y\in B_{\eta}(x_0)$ with $\eta=\min\{\eta_1,\eta_2\}$. Now, we have that $u_n(y)\geq u_{n_0}(y)$ and $v_n(y)\leq \vi^*(y)$ for all $n\geq n_0$. Thus
	\begin{equation*}
		u_n(y)-v_n(y)\geq u_{n_0}(y)-\vi^*(y)>\frac{\delta}{4}>0
	\end{equation*}
	for all $n\geq n_0$ and for all $y\in B_{\eta}(x_0)$. Then $B_{\eta}(x_0)\subseteq \{u_n>v_n\}$ for all $n\geq n_0$. 
	Using that $\li_2 v_n =0$ in $B_{\eta}(x_0)$ and taking the limit we obtain $\li_2 \vi=0$ in $B_{\eta}(x_0)$ which is  the same that $\li_2 \vi^*=0$ in $B_{\eta}(x_0)$ because $\vi$ is lower semicontinuous.
As a consequence, $\li_2 v_{n}\geq0$ in $B_{\eta}(x_{0})$. In particular, $\li_2 v_{n}\geq0$ in $B_{\eta}(x_{0})$ for all $n\geq n_{0}$. Then, by Proposition \ref{propv3}, $\li_{2} v_{\infty}\geq 0$ in $B_{\eta}(x_{0})$ in the viscosity sense. Since previously we just proved that $v_{\infty}$ is a $\li_{2}$-viscosity subsolution in $\Omega$, we obtain that $v_{\infty}$ is a $\li_{2}$-viscosity solution in $B_{\eta}(x_{0})$. Therefore, $v_{\infty}$ is a $\li_{2}$-viscosity solution in the set $\{v_{\infty}^{*}<u_{\infty}\}$.

Putting all together, since $\ui$ is lower semicontinuous, we obtain
	\begin{equation*}
		\begin{cases}
			\vi^*\leq \ui&\;\;\textrm{in}\;\;\Omega\\
				\vi^*=g&\;\;\textrm{in}\;\;\partial\Omega,\\
			\li_2 \vi \leq 0&\;\;\textrm{in}\;\;\Omega\;\;\textrm{in the viscosity sense,}\\
			\li_2 \vi = 0&\;\;\textrm{in}\;\;\{\vi^*<\ui \}\;\;\textrm{in the viscosity sense.}\\
		\end{cases}
	\end{equation*}
By uniqueness of the obstacle problem we get $\vi=v=\oou(\ld,\ui,g)$.

	Now, let us define $u=\ool(\lu,\vi,f)$. Using $v_n\leq \vi^*$, we have $u_{n-1}\leq u$. Then, taking the limit,  $\ui\leq u$. On the other hand we have that $\lu \ui\geq 0$ in $\Omega$ in viscosity sense due to Proposition \ref{propv3} and $\ui=f$ in $\partial\Omega$. Moreover, since $\ui$ is lower semicontinuous, taking $\limsup$ in the inequality $u_{n}\geq v_{n}$, we obtain $\ui\geq \vi^*$ in $\overline{\Omega}$.
	 Then, using the Lemma \ref{propv1}, we get $u\leq \ui$. Thus, we conclude that $\ui=u$.
	 
	Finally, we have that $u_{\infty}$ and $v_{\infty}$ are lower semicontinuous functions in $\bar{\Omega}$ by Lemma \ref{remark3}. Moreover, $\ui$ is continuous in the interior of $\{\ui=\vi^*\}$ because $\ui$ is a lower semicontinuous and, by definition, $\vi^*$ is an upper semicontinous function. This also implies that $\vi^{*}$ is a continuous function in the interior of  $\{\ui=\vi^*\}$. Then, $\vi=\vi^{*}$ there and therefore $\vi$ is a continuous function in the interior of the set $\{\ui=\vi\}$.
\end{proof}

\begin{remark}
	If $\ui$ or $\vi$ are continuous at $\partial\{\ui=\vi^{*}\}$, then, both functions $\ui$ and $\vi$ are continuous in $\bar{\Omega}$ and therefore the sequences defined in the previous theorem converge uniformly in the whole $\overline{\Omega}$.

In fact, without loss of generality, suppose that $\ui$ is continuous in $\partial\{\ui=\vi^*\}$. We will prove the above using the Comparison Principle for $\mathcal{L}_{1}$. Since $\mathcal{L}_{1}\ui=0$ in $\{\ui>\vi^{*}\}$ in the viscosity sense, we have in particular that $\ui^{*}$ is a $\mathcal{L}_{1}$-subsolution and $u_{\infty*}$ a $\mathcal{L}_{1}$-supersolution in $\{\ui>\vi^{*}\}$ in the viscosity sense. Moreover, since $f>g$ in $\partial\Omega$, the boundary $\partial\{\ui>\vi^{*}\}$ is the disjoint union of $\partial\Omega$ and $\partial\{\ui=\vi^{*}\}$. In $\partial\Omega$, we have $\ui^{*}=u_{\infty*}$ by construction of $u_{n}$. And in $\partial\{\ui=\vi^{*}\}$ we have also that $\ui^{*}=u_{\infty*}$ because we have supposed $\ui$  continuous across that bounday. Then, $\ui^{*}=u_{\infty*}$ in $\partial\{\ui>\vi^{*}\}$. Therefore, by the Comparison Principle of $\mathcal{L}_{1}$, $\ui^{*}\leq u_{\infty*}$ in $\overline{\{\ui>\vi^{*}\}}$ and therefore $\ui$ is continuous in $\overline{\{\ui>\vi^{*}\}}$. Moreover, as we  have seen in the proof of the above theorem,  $\ui$ is also continuous in $\{\ui>\vi^{*}\}$. Then $\ui$ is continuous in $\overline{\Omega}.$ As a consequence, using the latest hypothesis concerning the operator $\mathcal{L}_{2}$, $\vi$ is continuous in the whole $\overline{\Omega}$, because $\vi$ is the solution of the upper obstacle problem with continuous boundary datum and continuous obstacle. 

Finally, since the continuous functions $\ui$ and $\vi$ are the limit of the sequences of continuous functions $\{u_{n}\}$ and $\{v_{n}\}$ in a compact set $\bar{\Omega}$, the convergence is uniform.
\end{remark}

\begin{remark}  \label{rem77} As happens in the variational setting,
we also have here that the limit depends strongly on the initial function
from where we start the iterations. Moreover, the convergence of $(u_{n},v_{n})$ to a solution of the two membranes problem relies on the monotonicity of the sequences $\{u_{n}\}$ and $\{v_{n}\}$. This property comes from the fact that the initial function $v_{0}$ is a $\mathcal{L}_{2}$-viscosity subsolution. For this discussion in the variational context, we refer to the arguments given in Remark \ref{rem4}.
\end{remark}

	{\bf Acknowledgments}
	
	I. Gonzálvez was supported by the European Union's Horizon 2020 research and innovation programme under the Marie Sklodowska-Curie grant agreement No.\,777822, and by grants CEX2019-000904-S, PID2019-110712GB-I00, PID2020-116949GB-I00, and RED2022-134784-T, all four funded by MCIN/AEI/10.13039/ 501100011033 (Spain).
	
	A. Miranda and J. D. Rossi were partially supported by 
			CONICET PIP GI No 11220150100036CO
(Argentina), PICT-2018-03183 (Argentina) and UBACyT 20020160100155BA (Argentina).



\begin{thebibliography}{00}
	
	
\bibitem{ARS} A. Azevedo, J. F. Rodrigues and L. Santos. {\it The N-membranes problem for quasilinear degenerate systems}. Interfaces Free Bound., 7(3), (2005), 319--337.

\bibitem{BM} B. Barrios and M. Medina. {\it Equivalence of weak and viscosity solutions in fractional non-homogeneous problems}. Math. Ann., 381(3-4), (2021), 1979--2012.

		
		\bibitem{BRLibro} P. Blanc and J. D. Rossi. {Game Theory and Partial Differential Equations.}
De Gruyter Series in Nonlinear Analysis and Applications  Vol. 31. 2019.
ISBN 978-3-11-061925-6.
ISBN 978-3-11-062179-2 (eBook).

\bibitem{BlancR} P. Blanc and J. D. Rossi. 
{\it Obstacle problems and maximal operators}. Adv. Nonlinear Studies, 16(2), 
(2016), 355--362.
		
		\bibitem{cafa} L. A. Caffarelli. {\it The regularity of free boundaries in higher dimensions}. Acta Math.,
139(3-4) (1977), 155--184.

\bibitem{cafa2} L. A. Caffarelli. {\it The obstacle problem revisited}. J. Fourier Anal. Appl.,
4(4-5), (1998), 383--402.


\bibitem{Caffa1} 
L. Caffarelli, D. De Silva and O. Savin. {\it The two membranes problem for different operators.} Ann. l'Institut Henri Poincare C, Anal. non lineaire,
 34(4), (2017), 899--932.
 
 \bibitem{Caffa2} L. Caffarelli, L. Duque and H. Vivas. {\it The two membranes problem for fully nonlinear operators.} Discr. Cont. Dyn. Syst, 38(12), (2018), 6015--6027.



		
		\bibitem{CaCa} L. A. Caffarelli and X. Cabre, Fully nonlinear elliptic equations.
Amer. Math. Soc. Colloquium Publications, 43. Providence, RI, 1995.


\bibitem{CChV} S. Carillo, M. Chipot and G. Vergara-Caffarelli. {\it The N-membrane problem with nonlocal
constraints}, J. Math. Anal. Appl., 308(1), (2005), 129--139.


\bibitem{ChV} M. Chipot and G. Vergara-Caffarelli. {\it The N-membranes problem}, Appl. Math. Optim., 13(3),
(1985), 231--249.



\bibitem{CIL} M. G. Crandall, H. Ishii and P. L. Lions.
\textit{User guide to viscosity solutions of second order partial differential equations.}
Bull. Amer. Math. Soc., 27, (1992), 1--67.

\bibitem{toti} P. Daskalopoulos and P. M. N. Feehan. {\it $C^{1,1}$ Regularity for Degenerate Elliptic
Obstacle Problems in Mathematical Finance}. Preprint.


\bibitem{dBdPO} M. de Borbon, L. M. Del Pezzo and P. Ochoa. {\it Weak and viscosity solutions for non-homogeneous fractional equations in Orlicz spaces}. Adv. Differential Equations, 27(11-12), (2022), 735--780.

\bibitem{libroFinlan}
J. Heinonen, T. Kilpeläinen and O. Martio. Nonlinear potential theory of degenerate elliptic equations. Dover, (2006). 



   \bibitem{Hi} H. Ishii, {\it On the equivalence of two notions of weak solutions, viscosity solutions and distribution solutions}, Funkcial Ekvac. Ser. Int., 38 (1), (1995), 101--120.
   
     \bibitem{JLM} P.  Juutinen, P. Lindqvist, and J. J. Manfredi. {\it On the equivalence of viscosity solutions and weak solutions for a quasi-linear equation}. SIAM J. Math. Anal., 33(3), (2001), 699--717. 


\bibitem{Lewicka} M. Lewicka. A Course on Tug-of-War Games with Random Noise.
Introduction and Basic Constructions. Universitext book series. Springer, (2020).



\bibitem{MPRb} J. J. Manfredi, M. Parviainen and J. D. Rossi.
\textit{On the definition and properties of p-harmonious functions.}
Ann. Scuola Nor. Sup. Pisa, 11, (2012), 215--241.


\bibitem{Medina} M. Medina, P. Ochoa. {\it On viscosity and weak solutions for
			non-homogeneous p-Laplace equations.} Adv. Nonlinear Anal., 8, (2019), 
			468--481.
			
			\bibitem{Mon} R. Monneau. {\it A Brief Overview on The Obstacle Problem}. In: Casacuberta, C., Miró-Roig, R.M., Verdera, J., Xambó-Descamps, S. (eds) European Congress of Mathematics. Progress in Mathematics, vol 202. (2001).
Birkhäuser, Basel.
			
			\bibitem{PSSW} Y. Peres, O. Schramm, S. Sheffield and D. Wilson,
{\it Tug-of-war and the infinity Laplacian.} J. Amer. Math. Soc.,
22, (2009), 167--210.


\bibitem{PS} Y. Peres, S. Sheffield, {\it Tug-of-war with noise:
a game theoretic view of the $p$-Laplacian}, Duke Math. J., 145(1), (2008), 91--120.
			
			\bibitem{petro} A. Petrosyan, H. Shagholian, and N. Uraltseva, {\it Regularity of free boundaries in obstacle type problems}, American
Mathematical Society, Providence, RI, 2012.

   \bibitem{RWZ} J. Ren, J. Wu and M. Zheng. {\it On the equivalence of viscosity and distribution solutions of second-order PDEs with Neumann boundary conditions}. Stochastic Process. Appl.,
      130(2), (2020), 656--676. 


\bibitem{S} L. Silvestre. {\it The two membranes problem}, Comm. Partial Differential Equations, 30(1-3), (2005), 245--257.


\bibitem{VC} G. Vergara-Caffarelli, {\it Regolarita di un problema di disequazioni variazionali relativo a due
membrane}, Atti Accad. Naz. Lincei Rend. Cl. Sci. Fis. Mat. Natur. (8) 50, (1971), 659–-662
(Italian, with English summary).

\bibitem{Vivas} H. Vivas. {The two membranes problem for fully nonlinear local and nonlocal operators}. PhD Thesis dissertation. UT Austin. (2018). https://repositories.lib.utexas.edu/handle/2152/74361


	\end{thebibliography}
\end{document}